\documentclass[article,10pt, reqno]{amsart}
\usepackage[letterpaper,body={16.0cm,22.0cm}, mag=1000]{geometry}

\usepackage{amsmath,amsfonts,amssymb,upgreek}

\usepackage[breaklinks]{hyperref}
\usepackage{amscd}
\usepackage{mathtools}

\theoremstyle{thmrm}
\usepackage{tikz}
\usetikzlibrary{calc,decorations.markings}
\usepackage{tikz-cd}
\usetikzlibrary{cd}
\usepackage{graphicx}

\theoremstyle{plain}

\newtheorem{thm}{Theorem}[section]
\newtheorem{lemma}[thm]{Lemma}
\newtheorem{prop}[thm]{Proposition}
\newtheorem{cor}[thm]{Corollary}

\newtheorem*{thma}{Theorem A}
\newtheorem*{thmb}{Theorem B}

\theoremstyle{definition}
\newtheorem{remark}[equation]{Remark}

\newcommand{\dlabel}[1]{\ifmmode \text{\ttfamily \upshape [#1] } \else
{\ttfamily \upshape [#1] }\fi \label{#1}}

\newcommand{\B}{\operatorname{B} }
\newcommand{\C}{\operatorname{C} }

\newcommand{\Ho}{\operatorname{H} }

\newcommand{\RH}{\operatorname{RH} }

\newcommand{\Z}{\operatorname{Z} }

\newcommand{\im}{\operatorname{Im} }

\newcommand{\Id}{\operatorname{Id}}

\newcommand{\Aut}{\operatorname{Aut} }
\newcommand{\Autb}{\operatorname{Autb} }

\newcommand{\res}{\operatorname{res} }

\newcommand{\Ext}{\operatorname{Ext} }
\newcommand{\STExt}{\operatorname{STExt} }
\newcommand{\CExt}{\operatorname{CExt} }

\newcommand{\Fun}{\operatorname{Fun} }

\newcommand{\Ker}{\operatorname{Ker} }
\newcommand{\IM}{\operatorname{Im} }

\title{Cohomology, Extensions and Automorphisms of Skew Braces}

\author{Nishant}

\address{School of Mathematics, Harish-Chandra Research Institute, HBNI,
Chhatnag Road, Jhunsi, Allahabad - 211 019, INDIA}

\email{nishant@hri.res.in}

\author{Manoj K.~Yadav}

\address{School of Mathematics, Harish-Chandra Research Institute, HBNI, 
Chhatnag Road, Jhunsi, Allahabad - 211 019, INDIA}

\email{myadav@hri.res.in}


\begin{document}

\maketitle

\begin{abstract}
The second cohomology group of a left skew brace with coefficients in a trivial left brace with non-trivial actions is defined,  its connection with extensions of  a left skew brace by  a trivial braces is established and a Wells' like exact sequence relating the second cohomology group with inducible automorphisms of an extension of left skew  braces is constructed. 

\end{abstract}

\section{Introduction}

Classification of set-theoretic solutions of the Yang-Baxter equation, proposed by Drinfeld \cite{D92}, is a wide open problem.
In 2007, Rump \cite{WR07}  introduced the notion of braces in connection with  involutive non-degenerate  set-theoretic solutions of the Yang-Baxter equation. This concept was generalised by Guarnieri and Vendramin \cite{GV17} to the concept of skew braces to take into account  non-involutive non-degenerate solutions too.  An algebraic structure $(E, + , \circ)$ is said to be a left skew brace  if $(E, +)$  and $(E, \circ)$ are groups and, for all $a, b, c \in E$, the following compatibility condition holds:
\begin{equation*}\label{bcomp}
a \circ (b + c ) = (a \circ b) - a + (a \circ c).
\end{equation*}
A right skew brace can be defined analogously. In the present article, we shall consider only left skew  braces. A left skew brace is called a left brace if $(E, +)$ is an abelian group. The present definition of a left brace, which is  slightly different but equivalent to the one given by Rump,  is due to  Cedo, Jespers and Onkinski \cite{CJO}.

As it is well known now (see  \cite{CJO}), a left brace gives rise to a non-degenerate involutive set-theoretic solution of the Yang-Baxter equation (and vice-versa). There is also a bijection between left skew braces and non-degenerate  set-theoretic solutions of the Yang-Baxter equation (see \cite{GV17}). Not only this, braces have deep connections with many other algebraic structures; linear cycle sets, radical rings, Hopf-Galois extensions, bijective $1$-cocycles, to name some. So each new construction of a left (skew) brace contributes to the solutions of the Yang-Baxter equation and enriches the class of other related structures. Constructing new algebraic structures from the existing ones  is always an important and interesting problem. So is the case for  skew braces. We mention some such constructions. Notion of semidirect product of braces was introduced in \cite{WR08}, which was generalised to asymmetric product of braces  in \cite{CCS}. Asymmetric product of braces was further studied in \cite{BCJO19}, where the authors also studied the wreath product of braces. Matched product of braces was investigated in \cite{DB18}. Iterated matched product of braces was taken up in \cite{BCJO18}. Matched product of skew braces  was studied in \cite{CCS1, CCS2}. 

Our interest in this note is the cohomology and extension theory of skew braces. Extension theory on the one hand provides new constructions of the skew braces and on the other hand is expected to play a fundamental role in the classification of braces and skew braces. Vendramin \cite{V16} studied dynamical extensions of finite cycle sets and its relationship with dynamical co-cycles. The topic was further studied by Castelli, Catino, Miccoli and Pinto \cite{CCMP} for quasi-linear cycle sets. Dynamical extensions give rise to non-degenerate involutive multipermutation  solutions of the Yang-Baxter equation. Extensions of bijective $1$-cocycles were investigated by Ben David and Ginosar \cite{DG16}. Homology and cohomology theories for solution sets of the Yang-Baxter equations were developed  by Carter, Elhambadi and Saito \cite{CES}. Different homology theories for various structures related to solutions of   the Yang-Baxter equations were extensively investigated by Lebel and Vendramin \cite{LV17}. They also defined second cohomology group for a cycle set $X$ with coefficients in an abelian group $A$ and established  a bijective correspondence between the second cohomology group and equivalence classes of extensions of $X$ by $A$.  

The present article is based on several further advances in this direction: (a) A general homology and cohomology theory  was developed for linear cycle sets and left braces by Lebed and Vendramin \cite{LV16} with trivial actions, where the authors extensively explored close connections between the second cohomology of linear cycle sets and  extension theory.  (J.A. Guccione and J.J. Guccione \cite{GG21} further investigated the ideas of \cite{LV16} and used `perturbation lemma' for computing first and second cohomology groups and extensions of  cyclic linear cycles set acting trivially on an abelian group).    (b)  Bachiller \cite{DB18} considered extensions of left braces with non-trivial actions and characterised equivalent extension in the form of certain pairs of maps and relations among them. It turns out that these pairs of maps generalise the 2-cocycles from (a).

We develop extension theory for skew braces, define the second cohomology group of a skew brace acting on a trivial brace with non-trivial actions and establish connections between these.  This generalises results from (a) and (b). We then present an exact sequence connecting automorphism groups of  skew braces with second cohomology group.
Let $H:= (H, +, \circ)$ be a left skew brace and $I := (I, +)$ an abelian group viewed as a trivial left brace. Suppose that  $(H, \circ)$ acts on $(I, +)$ from left by an action $\nu$ and from right by  actons $\mu$ and  $\sigma$. The images of  $h \in H$ in $\Aut(I, +)$, the automorphism group of the group $(I, +)$, under $\nu$, $\mu$ and $\sigma$ are denoted by $\nu_h$, $\mu_h$ and $\sigma_h$ respectively. Of particular interest are the triplets of actions $(\nu, \mu, \sigma)$ which satisfy
\begin{eqnarray*}
\mu_{-h_{1} + (h_1 \circ h_3)}(y) &= & - \;\nu_{h_1 \circ (h_2 + h_3)}(\sigma_{h_2 + h_3}(\nu^{-1}_{h_1}(y)))+ \mu_{-h_1 + (h_1\circ h_3)}(\nu_{h_1 \circ h_2 }(\sigma_{h_2}(\nu^{-1}_{h_1}(y))))\\
& & + \; \nu_{h_1 \circ h_3}(\sigma_{h_3}(\nu^{-1}_{h_1}(y)))\\
\mu_{-h_1 + (h_1 \circ h_3)}(\nu_{h_1}(y))  & =& \nu_{h_1}(\mu_{h_3}(y)),
\end{eqnarray*}
for all $h_1, h_2, h_3 \in H$, $y \in I$. Such a triplet will be called good triplet of actions. For a given good triplet of actions, we define $\Ho_N^2(H, I)$, the second cohomology group of $H$ with coefficients in $I$ (see Section 3). Let $\Ext(H, I)$ denote the set of equivalence classes of extensions of $H$ by $I$ (see Section 2 for the definition). It turns out that 
$$\Ext(H, I) = \bigsqcup_{(\nu, \mu, \sigma)} \Ext_{(\nu, \mu, \sigma)}(H, I),$$
where the triplets $(\nu, \mu, \sigma)$ run over all good triplets of actions of $H$ on $I$ and $\Ext_{(\nu, \mu, \sigma)}(H, I)$ is the set of all equivalence classes of extensions of $H$ by $I$ whose corresponding good triplet of actions is $(\nu, \mu, \sigma)$ (see Corollary \ref{cor-sec3}). We can now state our first result
\begin{thma}\label{thma}
Let $H$ be a left skew brace, which acts on a trivial brace $I$ by a good triplet of actions $(\nu, \mu, \sigma)$. Then there exists a bijection between $\Ho_N^2(H, I)$ and $\Ext_{(\nu, \mu, \sigma)}(H, I)$.
\end{thma}

Let  $\Z^1_N(H,I)$ denote the group of all derivations from a left skew brace $H$ to a trivial brace $I$ as defined in Section 3. For an extension
$$\mathcal{E}: 0 \rightarrow I \rightarrow E \overset{\pi}\rightarrow H$$
 of $H$ by $I$, where $I$ is viewed as an ideal of $E$, we denote  by $\Autb_I(E)$ the set of all brace automorphisms of $E$ which normalize $I$. It follows that $\Autb_I(E)$ is a subgroup of $\Autb(E)$, the group of all brace automorphisms of $E$. For the extension $\mathcal{E}$, as explained in Section 3, one can associate a unique good triplet of actions $(\nu, \mu, \sigma)$. Let 
 $$\C_{(\nu, \mu, \sigma)} := \{ (\phi, \theta) \in \Autb(H) \times \Autb(I) \mid \nu_h=\theta^{-1}\nu_{\phi(h)}\theta, \mu_h = \theta^{-1}\mu_{\phi(h)}\theta \mbox{ and } \sigma_h = \theta^{-1}\sigma_{\phi(h)}\theta \}.$$
Our next result is an analog for skew braces of an exact sequence, called the fundamental exact sequence of Wells \cite{W71}, relating derivations, automorphisms and cohomology group of groups. For unexplained symbols in the following result see Section 5.

\begin{thmb}\label{thmb}
Let $\mathcal{E}: 0 \rightarrow I \rightarrow E \overset{\pi}\rightarrow H$ be a extension of a left skew brace  $H$ by a trivial left brace $I$ such that $[\mathcal{E}] \in \Ext_{(\nu, \mu, \sigma)}(H,I)$. Then we have the following exact sequence of groups 
$$0 \rightarrow \Z^1_N(H,I) \rightarrow \Autb_I(E) \stackrel{\rho(\mathcal{E})}{\longrightarrow} \C_{(\nu, \mu, \sigma)} \stackrel{\omega(\mathcal{E})}{\longrightarrow} \Ho^2_N(H,I),$$
where $\omega(\mathcal{E})$ is, in general,  only a derivation.
\end{thmb}

We remark that the fundamental exact sequence in group theoretical setting has been  revisited, reformulated and applied by many authors in the past, for example, see \cite{J07, JL10, PSY10, R13}. The exposition which we  present here is analogous to the one in \cite{JL10}. A unified treatment of the fundamental exact sequence of Wells with various applications is carried out in all fine details in \cite[Chapter 2]{PSY18}. A similar exact sequence for cohomology, extensions and automorphisms of quandles was constructed in \cite{BS20}.

We close this section with a quick layout of the article. Section 2 contains basic definitions and observations on left skew braces. In Section 3 we define the second cohomology group,  develop extension theory of left skew braces and establish  connection between these two, which enables us to prove  Theorem A  (as  Theorem \ref{gbij-thm sb}). An explicit construction is carried out for left braces of order $8$. In Section 4, generalising some ideas developed in \cite{LV16}, we define general cohomology theory of left skew braces in a particular case, but with non-trivial actions. The last section deals with automorphisms of skew braces, where various actions are exhibited which prepare a road map for the proof of Theorem B (as Theorem \ref{wells7 sb}). Finally, an application of Theorem B is exhibited, inducibility of a pair of automorphisms of skew braces is discussed in module theoretic language and a reduction argument for lifting and extension of automorphisms of nilpotent skew braces is presented.

\section{Preliminaries}

An algebraic structure $(E, + , \circ)$ is said to be a \emph{left skew brace} if $(E,  +)$ and $(E, \circ)$ are a group and,  the following compatibility condition holds:
\begin{equation}\label{bcomp}
a \circ (b  + c ) = a \circ b   - a +  a\circ c.
\end{equation}
for all $ a, b , c \in E$. 
Notice that the identity element $0$ of $(E,  +)$ coincides with the identity element  of $(E, \circ)$. A right skew brace can be defined analogously. In this article we'll only consider left skew braces and may take the liberty of using  `skew brace'  for  `left skew brace' at various places. 

For a left skew brace $E$ and $a \in E$, define a map $\lambda_a : E \to E$ by
$$\lambda_a(b) = - a + (a \circ b)$$
for all $b \in E$. The automorphism group  of a group $G$ is denoted by $\Aut(G)$. We have the following  result which was proved by Rump \cite{WR07} in linear cycle set settings.
\begin{lemma}
For each $a \in E$, the map $\lambda_a$ is an automorphism of $(E, +)$ and the map $\lambda : (E, \circ) \to \Aut(E, +)$ given by $\lambda(a) = \lambda_a$ is a group homomorphism.
\end{lemma}

A  sub skew brace $I$ of a left  skew brace $E$ is said to be a \emph{left ideal} of $E$ if $\lambda_a(y) \in I$ for all $a \in E$ and $y \in I$.  A left ideal of $E$ is said to be an \emph{ideal} if $(I, \circ)$ is  a normal subgroup of $(E, \circ)$.   An ideal $I$ of $E$ is said to be \emph{central} if $y \circ a = a \circ y = a + y$ for all $a \in A$ and $y \in I$. 

The following is an easy but important observation, which will be used several times in what follows.
\begin{lemma}
Let $E$ be a left skew brace. Then for all $a, b \in E$, the following hold:

(i) $a + b = a \circ \lambda^{-1}_a(b)$.

(ii) $a \circ b = a + \lambda_a(b)$.
\end{lemma}

Let $E_1$ and $E_2$ be two  left skew  braces. A map $f : E_1 \to E_2$ is said to be a \emph{brace homomorphism} if $f(a + b) = f(a) + f(b)$ and $f(a\circ b) = f(a) \circ f(b)$ for all $a, b \in E_1$. A one-to-one and onto brace homomorphism from $E_1$ to itself is called an \emph{automorphism} of $E_1$. The \emph{kernel} of a homomorphism $f : E_1 \to E_2$ is defined to be the subset $\{a \in E_1 \mid f(a) = 0\}$ of $E_1$. It turns out that $\Ker(f)$, the kernel of $f$, is an ideal of $E_1$. The set of all brace automorphisms of a left skew brace $E$, denoted by $\Autb(E)$, is a group.

Let $H$ and $I$ be two left skew braces. By an \emph{extension} of $H$ by $I$, we mean  a left skew brace $E$ with an exact sequence 
$$\mathcal{E} := 0 \to I \stackrel{i}{\to}  E \stackrel{\pi}{\to} H \to 0,$$ 
where $i$ and $\pi$, respectively,  are injective and surjective brace homomorphisms.
An extension $\mathcal{E}$ is said to be a \emph{central extension} if the image of $I$ under $i$ is  a central ideal of $E$. A set map $s : H \to E$ is called a \emph{set-theoretic section} of $\pi$ if $\pi(s(h)) = h$ for all $h \in H$ and $s(0) = 0$. The abbreviation `st-section' will be used for `set-theoretic section' throughout.

Let $\mathcal{E}$ and $\mathcal{E}'$ be two extensions of $H$ by $I$, that is,
$$\mathcal{E} : 0 \to I \stackrel{i}{\to}  E \stackrel{\pi}{\to} H \to 0$$
and $$\mathcal{E}' : 0 \to I \stackrel{i'}{\to}  E \stackrel{\pi'}{\to} H \to 0$$
are exact sequences of skew braces.
The extensions $\mathcal{E}$ and $\mathcal{E}'$ are said to be \emph{equivalent} if there exists a brace homomorphism $\phi : E \to E'$ such that the following diagram commutes:
$$\begin{CD}
 0 @>i>> I @>>> E @>{{\pi} }>> H  @>>> 0\\
 &&  @V{\text{Id}}VV @V{\phi}VV @ VV{ \text{Id}}V \\
 0 @>i'>> I @>>> E^\prime @>{{\pi^\prime} }>> H @>>> 0.
\end{CD}$$
The set of all equivalence classes of extensions of $H$ by $I$ is denoted by $\Ext(H, I)$.

Let $H$ be a left skew brace and $I$ an abelian group viewed as a trivial brace. Notice that $\Autb(I) = \Aut(I)$.  Then $I$ is said to be an \emph{$H$-module} if there exists a group homomorphism $\nu : (H, \circ) \to \Autb(I)$ and  group anti-homomorphisms $\mu : (H, +) \to \Autb(I)$ and  $\sigma : (H, \circ) \to \Autb(I)$. This means that $H$ acts on $I$ from left through $\nu$ and from right through $\mu$ and $\sigma$.  Such a  triplet $(\nu, \mu, \sigma)$ is called an action of $H$ on $I$. For an ideal $I$ of a left skew brace $E$, we get the natural exact sequence
$$ I \hookrightarrow E \overset{\pi} \twoheadrightarrow  H,$$
where $H \cong E/I$. Suppose that the ideal $I$ is a trivial brace. Let $s$ be an st-section of $\pi$. Then we can define a left action $\nu : H \to \Autb(I)$ and  right actions $\mu : (H, +) \to \Autb(I)$ and  $\sigma :  (H, \circ) \to \Autb(I)$ as given in Section 3 by \eqref{action1 sb},\eqref{action2 sb} and \eqref{action3 sb} respectively. Thus $I$ becomes an $H$-module. All $H$-modules $I$, in this paper, will be taken trivial braces. More distinctly, $I$ will always denote a trivial brace throughout.

Let $H$ and $H'$ be  left skew braces, and $I$ and $I'$ be  $H$ and $H'$ modules with $(\nu, \mu, \sigma)$ and  $(\nu', \mu', \sigma')$, respectively, as  actions.
Let $\alpha:H^\prime \rightarrow H$ and $\zeta: I \rightarrow I^\prime$ be brace homomorphisms.   The pair $(\alpha, \zeta)$ is said to be  \emph{compatible} with the triplet of actions $(\nu, \mu, \sigma)$ and $(\nu^\prime, \mu', \sigma')$ if the following diagram commutes for both left and right actions:

\begin{equation}\label{comp-homo}
\begin{CD}
 H@.\times    @.   I   @>{(\nu, \, \mu, \, \sigma)}>> I\\
 @A{\alpha}AA   @. @VV{\zeta}V @VV{\zeta}V\\
  {H^\prime}@.\times@. {I^\prime} @>{(\nu', \, \mu',  \,\sigma')}>>{I^\prime}.
\end{CD}
\end{equation}
More precisely, $(\alpha, \zeta)$ is  compatible with the actions, if 
$$\zeta(\nu_{\alpha(h^\prime)}(y))=\nu^\prime_{h^\prime}(\zeta(y)),$$
 $$\zeta(\mu_{\alpha(h^\prime)}(y))=\mu^\prime_{h^\prime}(\zeta(y))$$
 and
$$\zeta(\sigma_{\alpha(h^\prime)}(y))=\sigma^\prime_{h^\prime}(\zeta(y))$$
for all $h' \in H'$ and $y \in I$.
 Let $I$ and $I^\prime$ be two $H$-modules with  actions $(\nu, \mu, \sigma)$ and $(\nu^\prime, \mu', \sigma^\prime)$ respectively.  A map $\zeta: I \rightarrow I^\prime$ is called \emph{$H$-module homomorphism} if $(\Id, \zeta)$ is  compatible with the pairs of actions $(\nu, \mu, \sigma)$ and $(\nu^\prime, \mu', \sigma^\prime)$, where $\Id : H \to H$ is the identity homomorphism.


\section{Second  Cohomology and Extensions of Skew Braces}

Let $H$ be a left skew brace and $I$  an abelian group. Let  $\nu: (H, \circ) \rightarrow Aut (I, +)$ be a homomorphism and   $ \mu: (H, +) \rightarrow Aut(I, +)$  and $\sigma : (H, \circ) \rightarrow Aut(I, +)$ be anti-homomorphisms satisfying the following conditions
\begin{eqnarray*}
\mu_{-h_{1} + (h_1 \circ h_3)}(y) &= & - \;\nu_{h_1 \circ (h_2 + h_3)}(\sigma_{h_2 + h_3}(\nu^{-1}_{h_1}(y)))+ \mu_{-h_1 + (h_1\circ h_3)}(\nu_{h_1 \circ h_2 }(\sigma_{h_2}(\nu^{-1}_{h_1}(y))))\\
& & + \; \nu_{h_1 \circ h_3}(\sigma_{h_3}(\nu^{-1}_{h_1}(y)))\\
\mu_{-h_1 + (h_1 \circ h_3)}(\nu_{h_1}(y))  & =& \nu_{h_1}(\mu_{h_3}(y)),
\end{eqnarray*}
where $\nu_h, \mu_h$ and $\sigma_h$ denote the image of $h$ under $\nu, \mu$ and $\sigma$.
Such a triplet $(\nu, \mu, \sigma)$ will be called \emph{a good triplet of actions} of $H$ on $I$.

 For $i \ge 0$, $j \ge 1$, let $\Fun(H^{i+j}, I)$ denote the abelian group of all functions from $H^{i+j}$ to $I$ and $C_{N}^{ij} := C_{N}^{ij}(H, I)$  denote the subgroup of $\Fun(H^{i+j}, I)$ consisting of all functions  which  vanish on all degenerate tuples. A tuple $(h_1, \ldots, h_n) \in H^n$ is said to be \emph{degenerate} if $h_i = 0$ for at least one $i$, $1 \le i \le n$. Set  $C_N^{n} := \bigoplus_{i+j=n} C_N^{i,j}$. We are mainly interested in small values of $n$, that is, $n \le 3$.  To be more precise, we take $C_N^1 = C_N^{0,1}$, $C_N^2 = C_N^{0,2} \oplus C_N^{1,1}$ and $C_N^3 = C_N^{0,3} \oplus C_N^{1,2} \oplus C_N^{2,1}$.
 
  Define   $\partial^1 : C_N^1 \rightarrow C_N^2$ by  $ \partial^1 (\theta) =  (g, f)$, where $\theta \in C_N^1$ and for $h_1, h_2 \in H$,
\begin{eqnarray*}
 g(h_1, h_2) &=&  \theta(h_2) - \theta(h_1 + h_2) + \mu_{h_2}(\theta(h_1)),\\
f(h_1, h_2) &= & \nu_{h_1}(\theta(h_2)) - \theta(h_1 \circ h_2) + \nu_{h_1 \circ h_2}(\sigma_{h_2}(\nu_{h_1}^{-1}(\theta(h_1)))).
 \end{eqnarray*}
Next define  $\partial^2 : C^2_N \rightarrow C^3_N$ by
\begin{equation}\label{del2equ}
 \partial^2(f,g)= \big(\hspace{.1cm} \partial^{0,2}_v(g),\hspace{.1cm} \partial^{0,2}_h(g) - \; \partial^{1,1}_v(f), \; \partial^{1,1}_h(f)\big),
 \end{equation}
 where $(g, f) \in C_N^2$ and, for $h_1, h_2, h_3 \in H$,  $ \partial^{0,2}_v :  C_N^{0,2} \to C_N^{0,3}$, $ \partial^{0,2}_h :  C_N^{0,2} \to C_N^{1,2}$, $ \partial^{1,1}_v :  C_N^{1,1} \to C_N^{1,2}$ and $ \partial^{1,1}_h :  C_N^{1,1} \to C_N^{2,1}$ are defined by
\begin{eqnarray*}
 \partial^{0,2}_v(g)(h_1, h_2, h_3) &=&  g(h_2, h_3) - g(h_1 + h_2, h_3)+ g(h_1, h_2 + h_3) -\mu_{h_3}( g(h_1, h_2)),\\
\partial^{0,2}_h(g)(h_1, h_2, h_3) &=&   \nu_{h_1}(g(h_2, h_3)) +\mu_{h_1 \circ h_3}(g(h_1, - h_1))-\mu_{h_1 \circ h_3}(g(h_1 \circ h_2, - h_1))\\
& & -g((h_1 \circ h_2) - h_1, h_1 \circ h_3),\\
\partial^{1,1}_v(f)(h_1, h_2, h_3) &=& \mu_{- h_1+ (h_1 \circ h_3)}( f(h_1, h_2)) - f(h_1, h_2 + h_3)  + f(h_1, h_3),\\
\partial^{1,1}_h(f)(h_1, h_2, h_3) &=& \nu_{h_1}(f(h_2, h_3)) - f(h_1 \circ h_2, h_3) + f(h_1, h_2 \circ h_3) - \nu_{h_1 \circ h_2 \circ h_3} (\sigma_{h_3}(\nu^{-1}_{h_1 \circ h_2}f(h_1, h_2))).
\end{eqnarray*}

With this setting, we now prove
\begin{lemma}
$\im( \partial^1) \subseteq \Ker(\partial^2)$.
\end{lemma}
\begin{proof}
Let $\theta \in C_N^1$ such that  $\partial^1(\theta) = (g, f)$.  It is easy to see that  $\partial^{0,2}_v(g)   = 0 = \partial^{1,1}_h(f)$. So, it suffices to show that
 $\partial^{0,2}_h(g) -  \partial^{1,1}_v(f) = 0$.  For $h_1, h_2, h_3 \in H$, we have 
\begin{eqnarray*}
(\partial^{0,2}_h(g) - \partial^{1,1}_v(f))(h_1, h_2, h_3)  &=&   \nu_{h_1}\big(\theta(h_3)-\theta(h_2 +h_3) + \mu_{h_3}(\theta(h_2))\big)  +  \mu_{h_1 \circ h_3}\big( \theta(- h_1)\\
&& + \; \mu_{- h_1}(\theta(h_1))\big) - \mu_{h_1 \circ h_3}\big(\theta(- h_1) - \theta(h_1 \circ h_2 - h_1) + \mu_{-h_1}(h_1 \circ h_2)\big)\\
&&  - \;   \theta(h_1 \circ h_3) + \theta(h_1 \circ (h_2 + h_3)) -   \mu_{h_1 \circ h_3}(\theta(h_1 \circ h_2 - h_1))\\
&& - \; \mu_{-h_1 + (h_1 \circ h_3)}\big(\nu_{h_1}(\theta(h_2)) - \theta(h_1 \circ h_2) + \nu_{h_1 \circ h_2}(\sigma_{h_2}(\nu_{h_1}^{-1}(\theta(h_1))))\big)\\
&& + \; \nu_{h_1}(\theta (h_2+h_3))-\theta(h_1 \circ (h_2+h_3))\\
&& + \; \nu_{h_1 \circ (h_2 + h_3)}(\sigma_{h_2 + h_3}(\nu_{h_1}^{-1}(\theta(h_1)))) - \nu_{h_1}(\theta(h_3)) \\
&&  + \; \theta(h_1 \circ h_3)- \nu_{h_1 \circ h_3}(\sigma_{h_3}(\nu^{-1}_{h_1}(\theta(h_1))))
\end{eqnarray*}

Further solving we get 
\begin{eqnarray*}
(\partial^{0,2}_h(g) - \partial^{1,1}_v(f))(h_1, h_2, h_3) &=& \nu_{h_1}(\mu_{h_3}(\theta(h_2))) - \mu_{-h_1 + (h_1 \circ h_3)}(\nu_{h_1}(\theta(h_2)))\\
&& + \; \nu_{h_1 \circ (h_2 + h_3)}(\sigma_{h_2 + h_3}(\nu_{h_1}^{-1}(\theta(h_1)))) - \nu_{h_1 \circ h_3}(\sigma_{h_3}(\nu^{-1}_{h_1}(\theta(h_1)))\\
&& -\mu_{-h_1 +( h_1 \circ h_3)}(\nu_{h_1 \circ h_2}(\sigma_{h_2}(\nu_{h_1}^{-1}(\theta(h_1))))) + \mu_{-h_1 + (h_1 \circ h_3)}(\theta(h_1))
\end{eqnarray*} 
which is zero because $(\nu, \mu, \sigma)$ is a good triple of actions.
\hfill $\Box$

\end{proof}

Let $\Z_N^2(H, I) := \Ker (\partial^2)$ and $\B_N^2(H, I) := \im(\partial^{1})$.  Define $\Ho^2_N(H, I) :=  \Z_N^2(H, I)/\B_N^2(H, I)$ to be the second cohomology group of $H$ with co-efficients in $I$ corresponding to the given good triplet of actions $(\nu, \mu, \sigma)$. To avoid the cumbersome notation, underlying good triplet of action is omitted in $\Ho^2_N(H, I)$, which, indeed, heavily depends on it. Two different good triplets of actions of $H$ on $I$  may very well give rise to different cohomology groups as shown in the explicit constructions at the end of this section.

Elements of $\Z_N^2(H, I)$ and $\B_N^2(H, I)$ are, respectively,  called \emph{$2$-cocycles} and \emph{$2$-coboundaries}. Two $2$-cocycles $(\beta_1, \tau_1)$ and $(\beta_2, \tau_2)$ are said to be \emph{cohomologous} if $(\beta_1, \tau_1) - (\beta_2, \tau_2) \in \im(\partial^1)$; more precisely, if there exists a $\theta \in C_N^1$ such that 
$$ (\beta_1-\beta_2)(h_1, h_2)=  \theta(h_2) - \theta(h_1 + h_2) + \mu_{h_2}(\theta(h_1))$$
 and
$$(\tau_1-\tau_2)(h_1, h_2)=  \nu_{h_1}(\theta(h_2)) - \theta(h_1 \circ h_2) + \nu_{h_1 \circ h_2}(\sigma_{h_2}(\nu_{h_1}^{-1}(\theta(h_1))))$$
for all $h_1, h_2 \in H$. The elements of $Z^1_N(H, I) := \Ker(\partial^1)$ are  called \emph{derivations} or \emph{$1$-cocycles} or \emph{crossed homomorphisms}.

We remark that when $H$ is a left brace, i.e., $(H, +)$ is abelian, we can constitute the following zero-sequence:
$$C_N^0 \stackrel{\partial^0}{\longrightarrow} C_N^1  \stackrel{\partial^1}{\longrightarrow}  C_N^2 \stackrel{\partial^2}{\longrightarrow} C_N^3,$$
where $C_N^0 := I_{\nu} = \{y \in I \mid \nu_h(y) = y \mbox{ for all } h \in H\}$ and 
 $\partial^0 : C_N^0 \rightarrow C_N^1$ is defined by $\partial^0(y) =  f_y$ such that 
 $$f_y(h)= \nu_h(\sigma_h(y))-y.$$

The next result follows from the definitions.
\begin{lemma}\label{ext-lemma2}
Let $(g, f) \in C_N^2$. Then, for all $h_1, h_2 \in H$, the following hold:

(i) $g(h_1, 0) = g(0, h_2) = 0$.

(ii) $f(h_1, 0) = f(0, h_2) = 0$.
\end{lemma}

\vspace{.2in}

 Let $(\beta, \tau) \in C_N^2$. Define on $H \times I$, the following operations:
\begin{eqnarray}
(h_1, y_1) + (h_2, y_2) &=& \big(h_1 + h_2, \mu_{h_2}(y_1)+y_2+\beta(h_1, h_2)\big).\\
(h_1, y_1) \circ (h_2, y_2) &= & \big(h_1 \circ h_2, \nu_{h_1 \circ h_2}(\sigma_{h_2}(\nu_{h_1}^{-1}(y_1)))+\nu_{h_1}(y_2)+ \tau(h_1, h_2)\big).
\end{eqnarray}\

Then we have
\begin{thm}\label{ext-thm sb}
Let $(\nu, \mu, \sigma)$ be a good triplet of actions of $H$ on $I$. Then $H \times I$ takes the structure of a left skew brace under the operations defined in the preceding para if and only if $(\beta, \tau) \in Ker(\partial^2)$. 
\end{thm}
\begin{proof}
It is not difficult to show that the operations $'+'$ and $'\circ'$  in $H \times I$ are associative if and only if $\partial^{0,2}_v(\beta)$ and  $\partial^{1,1}_h(\tau)$, respectively, vanish.   Details are left for the reader. For $h_i \in H$ and $y_i \in I$, $1 \le i \le 3$, we have

\begin{eqnarray*}
(h_1, y_1) \circ ((h_2, y_2) + (h_3, y_3)) & =& (h_1, y_1) \circ (h_2+h_3,\mu_{h_3}(y_2)+y_3+\beta(h_2, h_3))\\
&=& (h_1 \circ (h_2 + h_3), \nu_{h_1 \circ (h_2+h_3)}(\sigma_{h_2+ h_3}(\nu_{h_1}^{-1}(y_1))))\\
& & + \; \nu_{h_1}(\mu_{h_3}(y_2)+y_3+\beta(h_2, h_3)) + \tau(h_1, h_2+h_3))
\end{eqnarray*}
and 
\begin{eqnarray*}
 (h_1, y_1) \circ (h_2, y_2) -(h_1, y_1)+(h_1, y_1) \circ (h_3, y_3)
&=& (h_1 \circ h_2, \nu_{h_1 \circ h_2}(\sigma_{h_2}(\nu_{h_1}^{-1}(y_1)))+\nu_{h_1}(y_2)\\
& & +  \; \tau(h_1, h_2)) +  (-h_1, -\mu_{-h_1}(y_1)+\beta(h_1, -h_1))\\
&& + \; (h_1 \circ h_3, \nu_{h_1 \circ h_3}(\sigma_{h_3}(\nu_{h_1}^{-1}(y_1))) \\
&& + \;  \nu_{h_1}(y_3)+ \tau(h_1, h_3))\\
&=& \big(h_1 \circ (h_2+h_3), \mu_{-h_1 + (h_1 \circ h_2)}(\nu_{h_1 \circ h_2}(\sigma_{h_2}(\nu^{-1}_{h_1}(y_1))))\big)\\
&  & + \; \nu_{h_1}(y_2) +\tau(h_1, h_2)-\mu_{-h_1}(y_1)-\beta(h_1, -h_1)\\
&& + \; \nu_{h_1 \circ h_3}(\sigma_{h_3}(\nu^{-1}_{h_1}(y_1))) + \nu_{h_1}(y_3)+\tau(h_1, h_3).
\end{eqnarray*}
Considering the fact that $(\nu, \mu, \sigma)$ is a  good triplet of actions, 
 $$(h_1, y_1) \circ ((h_2, y_2) + (h_3, y_3)) = (h_1, y_1) \circ (h_2, y_2) -(h_1, y_1)+(h_1, y_1) \circ (h_3, y_3)$$
   if and only if $\partial^{0,2}_h(g) -  \partial^{1,1}_v(f) = 0$. The proof is now complete.
\hfill $\Box$

\end{proof}

The left skew brace  structure on $H \times I$ (as established  in the preceding theorem) is an extension of $H$ by $I$, which we denote by the $7$-tuple data $(H, I, \nu, \mu, \sigma, \beta, \tau)$ with the  exact sequence
$$0  \to  I \stackrel{i}{\to} (H, I, \nu, \mu, \sigma,  \beta, \tau) \stackrel{\pi}{\to} H \to 0,$$ 
where $i(y) = (0, y)$ and $\pi(h, y) = h$. \emph{Throughout the paper, for such  extensions, we'll always take a fixed st-section $s : H \to (H, I, \nu, \mu, \sigma,  \beta, \tau)$ of $\pi$  given by $s(h) = (h, 0)$}.  \\

For a given left skew brace $H$ acting on a trivial brace $I$ such that the triplet of actions is a good triplet of actions of $H$ on $I$, we have shown the existence of an extension of $H$ by $I$. 
We'll now show that the existence of a good triplet of actions comes naturally  with a given extension of a left skew brace by a trivial brace. Let  $0 \to I \rightarrow E  \stackrel{\pi}{\rightarrow}  H \to 0$
be an extension of a left skew brace $H$ by a trivial brace $I$. For the ease of notation, we shall view $I$ as an ideal of $E$ through the given embedding. 
Let $s : H \rightarrow E$ be  any st-section of $\pi$. Then, for all $ h \in H$ and $y \in I$, we  define  $\nu, \sigma : (H, \circ) \rightarrow Aut(I, +)$ and $\mu :(H, +) \rightarrow Aut(I, +)$  by $h \mapsto \nu_h$, $h \mapsto \mu_h$ and $h \mapsto \sigma_h$, where
\begin{equation}\label{action1 sb}
\nu_h(y) := - s(h) + (s(h) \circ y) = \lambda_{s(h)}(y),
\end{equation}
 
 \begin{equation}\label{action2 sb}
 \mu_h(y) := - s(h) + y + s(h)
 \end{equation}
 
 and 
 
 \begin{equation}\label{action3 sb}
 \sigma_h(y) := s(h)^{-1} \circ y \circ s(h).
 \end{equation}
 It is not difficult to see  that, for a given st-section $s$,  $\nu$ is  a homomorphism and $\mu, \sigma $ are  anti-homomorphisms. 
 
 Next, for the extension  $0 \to I \rightarrow E  \stackrel{\pi}{\rightarrow}  H \to 0$,  define maps $\beta,  \tau : H \times H \rightarrow I$ by 
\begin{equation}\label{cocycle1 sb}
\beta(h_1, h_2) := - s(h_1 + h_2) + s(h_1) + s(h_2)
\end{equation}
and 
\begin{equation}\label{cocycle2 sb}
 \tau(h_1, h_2):= \nu_{h_1 \circ h_2} \big(s(h_1 \circ h_2)^{-1} \circ s(h_1) \circ s(h_2) \big) =  - s(h_1 \circ h_2) + (s(h_1) \circ s(h_2)),
 \end{equation}
 where $\nu$ is defined in \eqref{action1 sb}.
 
 We now establish a relationship among the actions (defined in \eqref{action1 sb} - \eqref{action3 sb}) and  $\beta$ and $\tau$ (defined in \eqref{cocycle1 sb} and \eqref{cocycle2 sb}). For $h_1, h_2, h_3 \in H$, on the one hand we have
\begin{eqnarray*}
s(h_1) \circ (s(h_2) + s(h_3)) & = &s(h_1) \circ \big(s(h_2 + h_3) + \beta(h_2, h_3)\big)\\
&= & s(h_1) \circ s(h_2 + h_3) - s(h_1) +  s(h_1) \circ \beta(h_2, h_3)\\
&=& s(h_1 \circ (h_2 + h_3)) + \tau(h_1, h_2 + h_3)+\nu_{h_1}(\beta(h_2, h_3))
\end{eqnarray*}
and on the other hand 
\begin{eqnarray*}
s(h_1) \circ (s(h_2) + s(h_3)) & = & (s(h_1) \circ s(h_2)) - s(h_1) + (s(h_1) \circ s(h_3))\\
&=& s(h_1\circ h_2) + \tau(h_1, h_2) + s(-h_1) - \beta(h_1, - h_1) + s(h_1 \circ h_3) + \tau(h_1, h_3)\\
& =& s(h_1\circ h_2 - h_1) + \beta(h_1 \circ h_2, -h_1) + s(h_1 \circ h_3)\\
 & &+ \; \mu_{-h_1 + (h_1 \circ h_3)}(\tau(h_1, h_2)) - \mu_{h_1 \circ h_3}(\beta(h_1, - h_1)) + \tau(h_1, h_3)\\
 &=& +\; s(h_1 \circ (h_2 + h_3)) +  \beta(h_1 \circ h_2 - h_1, h_1 \circ h_3) + \mu_{h_1 \circ h_3}(\beta(h_1 \circ h_2, -h_1))\\
&& + \;  \mu_{-h_1 + (h_1 \circ h_3)}(\tau(h_1, h_2))  - \mu_{h_1 \circ h_3}(\beta(h_1, - h_1)) + \tau(h_1, h_3).
\end{eqnarray*}
Hence,  comparing right hand sides of the preceding equations, we get
\begin{eqnarray}\label{mutual1}
\mu_{- h_1 + (h_1 \circ h_3)}(\tau(h_1, h_2)) - \tau(h_1, h_2 + h_3) +\tau(h_1, h_3) &=& \nu_{h_1}(\beta(h_2, h_3)) + \mu_{h_1 \circ h_3}(\beta(h_1, - h_1)) \nonumber \\ 
& & - \; \mu_{h_1 \circ h_3}(\beta(h_1 \circ h_2, - h_1)) \nonumber  \\ 
& &  - \; \beta\big((h_1 \circ h_2) - h_1, h_1 \circ h_3\big).
\end{eqnarray}

With this setting, we have
\begin{prop}\label{well-def-act-coc}
 Let $0 \to I \stackrel{}{\to} E \stackrel{\pi}{\rightarrow}  H \to 1$ be an extension of a left skew brace $H$ by a trivial brace $I$. Then the following hold:
 
 (1) The actions $\nu$, $\mu$ and $ \sigma$ are independent of the choice of an st-section. Moreover, the triplet $(\nu, \mu,\sigma)$ is a good triplet of actions of $H$ on $I$.
 
(2) Equivalent extensions have the same triplet of actions.

(3) The pair $(\beta, \tau) \in \Ker(\partial^2)$, where $\partial^2$ is defined in \eqref{del2equ}.  Moreover, the cohomology class of $(\beta, \tau)$ is independent of the choice of an st-section of $\pi$.
\end{prop}
\begin{proof}
That $\mu$ and $\sigma$ are independent of an st-section, follows from group theoretic constructions (see \cite[Section 11.1]{DJSR}). We'll only show that $\nu$ is independent of an st-section. Let $s_1$ and $s_2$ be two st-sections and $\nu$ and $\nu^\prime$, respectively,  be the corresponding actions. We know that for each $h \in H$ there exist $y_h \in I$ such that  $s_2(h) = y_h \circ s_1(h)$. Then we have  
\begin{equation*}
\begin{split}
\nu^\prime_h(y) & = \lambda_{s_2(h)}(y)\\
&=\lambda_{y_h \circ s_1(h)}(y)\\
&=\lambda_{y_h}(\lambda_{s_1(h)}(y))\\
&=\nu_h(y) \ \  \mbox{  (since $I$ is a trivial brace)},
\end{split}
\end{equation*}
which shows that $\nu$ is independent of a given st-section.

Notice that $s(h) \circ y = s(h) + \nu_h(y)$ and $s(h) + y = s(h) \circ \nu^{-1}_h(y)$ for $h \in H$ and $y \in I$.  For $h_i \in H$, $y \in I$, $ 1 \le i \le 3$, on the one hand we have
\begin{eqnarray*}
(s(h_1) + y) \circ \big(s(h_2)  +  s(h_3)\big) &=& (s(h_1) \circ \nu^{-1}_{h_1}(y)) \circ \big(s(h_2 + h_3) + \beta(h_2, h_3)\big)\\
&=& s(h_1) \circ s(h_2 + h_3) \circ \sigma_{h_2 \circ h_3}(\nu_{h_1}^{-1}(y)) \circ \nu^{-1}_{h_2+h_3}(\beta(h_2, h_3))\\
&=& s(h_1 \circ (h_2 + h_3)) +  \tau(h_1, h_2 + h_3)+ \nu_{h_1 \circ (h_2 + h_3)}(\sigma_{h_2 + h_3}(\nu^{-1}_{h_1}(y)))\\
&& + \; \nu_{h_1}\big(\beta(h_2, h_3))
\end{eqnarray*}
 and on the other hand
 \begin{eqnarray*}
(s(h_1) + y) \circ \big(s(h_2) + s(h_3)\big) &=& (s(h_1) + y) \circ s(h_2) - (s(h_1) + y) +   (s(h_1) + y_1) \circ s(h_3)\\
&=& (s(h_1) \circ \nu^{-1}_{h_1}(y)) \circ s(h_2)  - y + s(-h_1) - \beta(h_1, -h_1)\\
& &  +\;  (s(h_1) \circ \nu^{-1}_{h_1}(y)) \circ s(h_3)\\
&=& s(h_1 \circ h_2) + \tau(h_1, h_2) + \nu_{h_1 \circ h_2 }(\sigma_{h_2}(\nu^{-1}_{h_1}(y))) + - y + s(-h_1)\\
&& - \; \beta(h_1, -h_1) + s(h_1 \circ h_3) + \tau(h_1, h_3) + \nu^{-1}_{h_1 \circ h_3 }(\sigma_{h_3}(\nu^{-1}_{h_1}(y)))\\
&=& s(h_1 \circ h_2) + s(-h_1) + s(h_1 \circ h_3) + \mu_{h_1\circ h_3}(\mu_{-h_1}(\tau(h_1, h_2)))\\
& & +  \; \mu_{-h_1 + (h_1\circ h_3)}(\nu_{h_1 \circ h_2 }(\sigma_{h_2}(\nu^{-1}_{h_1}(y)))) - \mu_{-h_1 + (h_1 \circ h_3)}(y)\\
&&  - \; \mu_{h_1 \circ h_3}(\beta(h_1, -h_1)) + \tau(h_1, h_3) + \nu_{h_1 \circ h_3 }(\sigma_{h_3}(\nu^{-1}_{h_1}(y)))\\
&=&  s(h_1 \circ (h_2 + h_3)) + \beta(h_1 \circ h_2 - h_1, h_1 \circ h_3)\\
&&  +\; \mu_{h_1 \circ h_3}(\beta(h_1 \circ h_2 - h_1)) +  \mu_{-h_1 + (h_1\circ h_3)}(\tau(h_1, h_2))\\
& & +  \; \mu_{-h_1 + (h_1\circ h_3)}(\nu_{h_1 \circ h_2 }(\sigma_{h_2}(\nu^{-1}_{h_1}(y)))) - \mu_{-h_1 + (h_1 \circ h_3)}(y)\\
&&  - \; \mu_{h_1 \circ h_3}(\beta(h_1, -h_1)) + \tau(h_1, h_3) + \nu_{h_1 \circ h_3 }(\sigma_{h_3}(\nu^{-1}_{h_1}(y))).
 \end{eqnarray*}
Equating the right hand sides of the preceding two expressions and using \eqref{mutual1}, we get
\begin{eqnarray*}
\mu_{-h_{1} + (h_1 \circ h_3)}(y) &= & - \;\nu_{h_1 \circ (h_2 + h_3)}(\sigma_{h_2 + h_3}(\nu^{-1}_{h_1}(y)))+ \mu_{-h_1 + (h_1\circ h_3)}(\nu_{h_1 \circ h_2 }(\sigma_{h_2}(\nu^{-1}_{h_1}(y))))\\
& & + \; \nu_{h_1 \circ h_3}(\sigma_{h_3}(\nu^{-1}_{h_1}(y))).
\end{eqnarray*}
Finally, an easy direct computations shows that
$\mu_{-h_1 + (h_1 \circ h_3)}(\nu_{h_1}(y))   = \nu_{h_1}(\mu_{h_3}(y)).$
Hence the triplet $(\nu, \mu,\sigma)$ is a good triplet of actions, which completes the proof of assertion (1).
 
For assertion (2), let $E$ and $E^\prime$ be two equivalent extensions of $H$ by $I$. Then there exists a brace homomorphism $\phi : E \rightarrow  E^\prime$ such that the following  diagram commutes:
$$\begin{CD}
 0 @>>> I @>>> E @>{{\pi} }>> H  @>>> 0\\
 &&  @V{\text{Id}}VV @V{\phi}VV @VV{ \text{Id}}V \\
 0 @>>> I @>>> E^\prime @>{{\pi^\prime} }>> H @>>> 0.
\end{CD}$$
Let $s$  be an st-section of  $ \pi$. By the commutativity of the diagram, it follows that $s': H \to E'$, given by  $s^\prime(h) := \phi(s(h))$ for all $ h \in H$, is an st-section of $\pi'$, and $\phi(y)=y$ for all $ y \in I$. Let $(\nu, \mu,  \sigma)$ and $(\nu', \mu', \sigma')$ be actions corresponding to $s$ and $s^\prime$ respectively. Then, using the fact that $\phi$ is a homomorphism, we have
\begin{equation*}
\begin{split}
\nu'_{h}(y)& = - s^\prime(h) + ( s^\prime(h) \circ y)\\
& =  - \phi( s(h)) + ( \phi( s(h) \circ y)\\
& =   - \phi\big(s(h) + ( s(h) \circ y)\big)\\
&=  \phi(\nu_h(y))\\
&=  \nu_h(y)
\end{split}
\end{equation*}
for all $h \in H$ and $y \in I$. Similarly one can easily show that $\mu'_h(y) = \mu_h(y)$ and $\sigma'_h(y) = \sigma_h(y)$ for all $h \in H$ and $y \in I$. Thus,  $(\nu, \mu,  \sigma)$ and  $(\nu', \mu', \sigma')$  are the same actions, which establishes assertion (2).

We now prove assertion (3). Let $s$ be an st-section of $\pi$.  That $\partial^{0,2}_v(\beta) = 0$ and  $\partial^{1,1}_h(\tau) = 0$ follows from the associativity of $'+'$ and $'\circ'$ in $E$ respectively.  Details are left for the reader. That $\partial_h^{0,1}(\beta) - \partial_v^{1,1}(\tau) = 0$ follows from \eqref{mutual1}. This shows that $(\beta, \tau) \in \Ker(\partial^2)$. Finally, again assume that  $s_1$ and $ s_2$ are two st-sections of $\pi$. Let $(\beta, \tau)$ and $ ( \beta^\prime, \tau^\prime)$ be $2$-cocycles corresponding to $s$ and $s_1$ respectively. Notice that  $- s_1(h) + s_2(h) \in I$. Define a map $\theta : H \to I$ by  $\theta(h) = - s_1(h) + s_2(h)$ for all $h \in H$. A straightforward computation then shows that the $2$-cocycles $(\beta, \tau)$ and $(\beta^\prime, \tau^\prime)$ differ by $\partial^1(\theta)$; meaning cohomologous. This completes the proof.
\hfill $\Box$

\end{proof}

Let $\Ext(H,I)$ denote the set of  equivalence classes of all  extensions of $H$ by $I$.  Equivalence class of an extension $\mathcal{E} : 0 \to I \to E \to H \to 0$ is denoted by $[\mathcal{E}]$. As a consequence of the preceding proposition, it follows that each equivalence class of extension of $H$ by $I$ admits a unique good triplet of actions $(\nu, \mu, \sigma)$  of $H$ on $I$, which we call the \emph{corresponding triplet  of actions}. Let $\Ext_{(\nu, \mu, \sigma)}(H, I)$ denote the equivalence class of those extensions of $H$ by $I$ whose corresponding triplet of actions is $(\nu, \mu, \sigma)$.  We can easily establish
 
  \begin{cor}\label{cor-sec3}
 $\Ext(H, I) = \bigsqcup_{(\nu, \mu, \sigma)} \Ext_{(\nu, \mu,\sigma)}(H, I)$.
 \end{cor}

We are now ready to prove the main result of this section (Theorem A).
\begin{thm}\label{gbij-thm sb}
Let $H := (H, +, \circ)$ be a left  skew brace and $I := (I, +)$ an abelian group viewed as a trivial brace. Let $(\nu, \mu, \sigma)$ be a good triplet of actions of $H$ on $I$. Then there is a bijection between $\Ho^2_N(H, I)$ and $\Ext_{(\nu, \mu, \sigma)}(H, I)$.
\end{thm}
 \begin{proof}
Let $( \beta, \tau) \in \Z^2_N(H, I)$. It now follows from Theorem \ref{ext-thm sb} that the pair $(\beta, \tau)$ gives rise to an extension $(H, I, \nu, \mu, \sigma,  \beta, \tau)$.  Let $(\beta_1, \tau_1), (\beta_2, \tau_2) \in \Z^2_N(H, I)$ be cohomologous. Then there exists  $~\theta \in C^1_N$ such that
 \begin{equation}\label{cc1 sb}
 (\beta_1-\beta_2)(h_1, h_2)=  \theta(h_2) - \theta(h_1  + h_2) + \mu_{h_2}(\theta(h_1))
\end{equation}
 and
\begin{equation}\label{cc2 sb}
(\tau_1-\tau_2)(h_1, h_2)=  \nu_{h_1}(\theta(h_2)) - \theta(h_1 \circ h_2) + \nu_{h_1 \circ h_2}(\sigma_{h_2}(\nu_{h_1}^{-1}(\theta(h_1)))),
 \end{equation}
 for all $h_1, h_2 \in H$.
 Let $(H, I, \nu, \mu, \sigma,  \beta_1, \tau_1)$ and $(H, I, \nu, \mu, \sigma,\beta_2, \tau_2)$ be the corresponding extensions for $(\beta_1, \tau_1)$ and $(\beta_2, \tau_2)$ respectively.  Define  a map $\phi : (H, I, \nu, \mu, \sigma, \beta_1, \tau_1) \to (H, I, \nu, \mu, \sigma, \beta_2, \tau_2)$ by
 $$\phi(h, y) := (h, y + \theta(h)).$$

 Consider the following diagram:
$$\begin{CD}
0 @>>> I @>i>>  (H, I, \nu, \mu, \sigma, \beta_1, \tau_1)@>{{\pi} }>> H @>>> 0\\ 
&& @V{\text{Id}} VV@V{\phi} VV @V{\text{Id} }VV \\
0 @>>> I @>i^\prime>>  (H, I, \nu, \mu, \sigma, \beta_2, \tau_2) @>{{\pi^\prime} }>> H  @>>> 0.
\end{CD}$$
Notice that the linearity of $\phi$ in $'+'$ and $'\circ'$ follows from the equations \eqref{cc1 sb} and \eqref{cc2 sb} respectively; hence $\phi$ is a brace homomorphism.  It is now obvious that the preceding diagram is commutative. Thus the extensions $(H, I, \nu, \mu,\sigma, \beta_1, \tau_1)$ and $(H, I, \nu, \mu, \sigma, \beta_2, \tau_2)$ are equivalent, and therefore the map $\psi : \Ho^2_N(H, I) \to \Ext_{(\nu, \mu,  \sigma)}(H, I)$ given by
$$\psi([(\beta, \tau)]) = [(H, I, \nu,  \mu, \sigma, \beta, \tau)]$$
is well defined.

We now show that $\psi$ is surjective.  Let $\mathcal{E} : 0 \rightarrow I  \rightarrow E \overset{\pi} \rightarrow H \rightarrow 0$ be any extension whose corresponding triplet of actions is $(\nu, \mu, \sigma)$. Let $s : H \to E$ be an st-section of $\pi$. Then, by Proposition \ref{well-def-act-coc}(3),  it follows that $(\beta, \tau)$ (as defined in \eqref{cocycle1 sb} and \eqref{cocycle2 sb}) lies in  $\Z_N^2(H, I)$.  Applying $\psi$, we have $\psi([(\beta, \tau)]) = [(H, I, \nu, \mu, \sigma, \beta, \tau)]$. Now the surjectivity of $\psi$ is equivalent to showing that  $\mathcal{E}$ and $(H, I, \nu, \mu, \sigma, \beta, \tau)$ are equivalent. More precisely, the cohomology class $[(\beta, \tau)]$ of $(\beta, \tau)$ will be a pre-image of the extension $\mathcal{E}$.  We remark that this mechanism is independent of the choice of an st-section $s$ of $\pi$, because whatever $s$ we start with, we finally get an extension equivalent to $\mathcal{E}$. So, for the surjectivity of $\psi$, we only need to
establish the commutativity of the diagram
$$\begin{CD}
0 @>>> I @>i>>E @>{{\pi} }>> H @>>> 0\\ 
&& @V{\text{Id}} VV@V{\phi} VV @V{\text{Id} }VV \\
0 @>>> I @>i^\prime>> (H, I, \nu, \mu, \sigma, \beta, \tau) @>{{\pi^\prime} }>> H @>>> 0,
\end{CD}$$
in which $\phi$ is a brace homomorphism. As mentioned above, $\pi'(h, y) = h$ and $s'$, given by $s'(h) := (h, 0)$, is an st-section of $\pi'$, which we'll use here. Notice that every $g \in E $ can be uniquely written as  $g= s(h) + y$ for some  $ h \in H$ and $ y \in I$. Define $\phi : E \rightarrow (H, I, \nu, \mu, \sigma, \beta, \tau)$ by 
$$\phi(s(h) + y) = (h, y).$$ 
It is easy to see that  $\phi$ is linear under $'+'$. Notice that $s(h) \circ y = s(h) + \nu_h(y)$ and $s(h) +  y = s(h) \circ \nu^{-1}_h(y)$ for all $h \in H$ and $y \in I$. Let $s(h_1) + y_1$ and $s(h_2) + y_2$ be two elements of $E$. Then
\begin{eqnarray*}
\phi\big((s(h_1) + y_1) \circ (s(h_2) + y_2)\big) &=& \phi\big(s(h_1) \circ \nu^{-1}_{h_1}(y_1) \circ s(h_2) \circ \nu^{-1}_{h_2}(y_2)\big)\\
&=& \phi\big(s(h_1)\circ s(h_2)\circ\sigma_{h_2}(\nu^{-1}_{h_1}(y_1))\circ \nu^{-1}_{h_2}(y_2)\big)\\
&=& \phi\big(s(h_1 \circ h_2) \circ \nu^{-1}_{h_1 \circ h_2}(\tau(h_1, h_2))\circ\sigma_{h_2}(\nu^{-1}_{h_1}(y_1))\circ \nu^{-1}_{h_2}(y_2)\big)\\
&=&  \phi\big(s(h_1 \circ h_2) + \nu_{h_1 \circ h_2} \big(\nu^{-1}_{h_1 \circ h_2}(\tau(h_1, h_2)) + \sigma_{h_2}(\nu^{-1}_{h_1}(y_1)) + \nu^{-1}_{h_2}(y_2)\big)\big)\\
&=& \big(h_1 \circ h_2, \tau(h_1, h_2) + \nu_{h_1 \circ h_2}(\sigma_{h_2}(\nu^{-1}_{h_1}(y_1))) + \nu_{h_1}(y_2)\big)\\
&=& (h_1, y_1)\circ(h_2, y_2)\\
&=& \phi(s(h_1) + y_1) \circ \phi(s(h_2) + y_2).
\end{eqnarray*}
Hence $\phi$ is a brace homomorphism. Also $\phi(y) = (0, y) = i'( y)$ and 
$$\pi(s(h) + y)= h = \pi'(h, y) = \pi^\prime(\phi(s(h) + y)),$$
which establishes the  commutativity of the preceding diagram. 

Finally, we proceed to show the injectivity of  $\psi$. Let $(\beta_1, \tau_1), (\beta_2, \tau_2) \in \Z_N^2(H, I)$ such that $\psi([(\beta_1, \tau_1)]) = \psi([(\beta_2, \tau_2)])$. More precisely, we have the following commutative diagram:
$$\begin{CD}
0 @>>> I @>>>  (H, I, \nu, \mu, \sigma, \beta_1, \tau_1)@>{{\pi} }>> H @>>> 0\\ 
&& @V{\text{Id}} VV@V{\phi} VV @V{\text{Id} }VV \\
0 @>>>I @>>>  (H, I, \nu, \mu, \sigma, \beta_2, \tau_2) @>{{\pi^\prime} }>> H @>>> 0,
\end{CD}$$
where $\phi$ is a brace homomorphism.  Let $s$ and $s'$ be the st-sections of $\pi$ and $\pi'$ given by $s(h) = (h, 0)$ and $s'(h) = (h, 0)$, respectively, for all $h \in H$. It now follows from the commutativity of the diagram that, for each $h \in H$, $\phi(h, 0) = (h, y_h)$ for some $y_h \in I$. Notice that for a given $h \in H$, $y_h$ is unique. Define $\theta : H \to I$ by $\theta(h) = y_h$. Since $s(0) = (0,0)$, it follows that $\theta \in C_N^1$. Comparing second components in the final computations of $\phi((h_1, y_1) + (h_2, y_2)) = \phi(h_1, y_1) + \phi(h_2, y_2)$ and  $\phi((h_1, y_1) \circ (h_2, y_2)) = \phi(h_1, y_1) \circ \phi(h_2, y_2)$, we, respectively, get
$$ (\beta_1-\beta_2)(h_1, h_2)=  \theta(h_2) - \theta(h_1 + h_2) + \mu_{h_2}(\theta(h_1))$$
 and
$$(\tau_1-\tau_2)(h_1, h_2)=  \nu_{h_1}(\theta(h_2)) - \theta(h_1 \circ h_2) + \nu_{h_1 \circ h_2}(\sigma_{h_2}(\nu_{h_1}^{-1}(\theta(h_1))))$$
for all $h_1, h_2 \in H$. 
This simply means that $(\beta_1, \tau_1)$ and $(\beta_2, \tau_2)$ are cohomologous, and the proof is complete.
\hfill $\Box$

\end{proof}

Let $H$ be a left skew brace and $I$  a trivial brace such that $H$ acts trivially on $I$, that is,  all $\nu_h$, $\mu_h$ and $\sigma_h$ are trivial automorphism of $(I, +)$ for all $h \in H$. It is immediate, in this case, that $(\nu, \mu, \sigma)$ is a good triplet of actions. Let $[(\beta, \tau)] \in \Ho_N^2(H, I)$ and $E:=(H,I, \nu, \mu, \sigma, \tau,\beta)$ be the corresponding extension of $H$ by $I$. It turns out that the image of $I$ in $E$ is a central ideal of $E$. On the other hand, if 
$$\mathcal{E} : 0 \to I \longrightarrow E \stackrel{\pi}{\longrightarrow} H \to 0$$
 is a central extension, then   $\mathcal{E} \in \Ext_{(\nu, \mu, \sigma)}(H, I)$, where $(\nu, \mu, \sigma)$ is a trivial action of $H$ on $I$.  Let the set of all equivalence classes of central extension of $H$ by $I$ be denoted by $\CExt(H, I)$.  Then, as a special case of Theorem \ref{gbij-thm sb}, we get the following brace theoretic analog of  \cite[Theorem 5.8]{LV16}.
 
 \begin{cor}
Let $H := (H, +, \circ)$ be a left  brace which acts trivially on a trivial brace $I := (I, +)$.  Then there is a bijection between $\Ho^2_N(H, I)$ and $\CExt(H, I)$.
\end{cor}
  
We conclude this section with explicit computation of $\Ho_N^2(H,I)$, where $H = \mathbb{Z}/2\mathbb{Z}$ and $I=\mathbb{Z}/2\mathbb{Z} \times \mathbb{Z}/2\mathbb{Z}$ with two good triplets of actions; one non-trivial and other trivial. It is easy to see that the triplet $(\nu, \mu, \sigma)$ with $\nu_{h}(a,b)=(a+b+h, b)$, $\mu_h(a, b) = (a, b)$ and $\sigma_h(a,b)=(a,b+a+h)$  is a good triplet of actions, where $h \in H$ and $a, b \in I$. This triplet of actions was considered by Bachiller \cite{DB18}. Let $(\beta, \tau)$  be a  $2$-cocycle. Notice that both $\beta$ as well as $\tau$ are determined by their values at $(1,1) \in H \times H$. An easy computation then reveals that the following are the only choices for $\big(\beta(1,1), \tau(1,1)\big)$: 
$$\{\big((0,0), (0,0)\big), \big((0,0), (0,1)\big), \big((1,0), (0,0)\big)\big((1,0), (0,1)\big)\}.$$
For clarity, let us take these cocycles as $(\beta_1, \tau_1), (\beta_1, \tau_2), (\beta_2, \tau_1), (\beta_2, \tau_2)$ respectively. If we define $\theta(0)=(0,0)$ and $\theta(1)=(0,1)$, then $\theta \in C_N^1$; further $(\beta_1, \tau_2)=\partial^1(\theta)$, that is, $(\beta_1, \tau_2)$ is a $2$-coboundary. We claim that none of $(\beta_2, \tau_1)$ and $(\beta_2, \tau_2)$ is a $2$-coboundary, but these are cohomologous. Contrarily, assume that $(\beta_2,\tau_2)$ is a $2$-coboundary. Then there exists $\theta : C_N^1$ such that 
 $\partial^1(\theta)=(\beta_2,\tau_2)$.  So $\theta(0)=(0,0)$ and  let $\theta(1)=(y_1, y_2)$. Now 
\begin{eqnarray*}
\beta_2(1,1) &=&  \theta(1)-\theta(0)+ \theta(1)\\
& =&  (y_1, y_2) + (y_1, y_2)\\
& =& (0, 0),
\end{eqnarray*}
which is not possible, since $\beta_2(1,1) = (1, 0)$. Hence $(\beta_2,\tau_2)$ is not a $2$-couboundary. The same computation also shows that $(\beta_2, \tau_1)$ too is not a $2$-coboundary.
Finally,  notice that $(\beta_2, \tau_2)-(\beta_2, \tau_1) = (\beta_1, \tau_2)$, which we have already shown to be a  $2$-coboundary. Hence $(\beta_2, \tau_2)$ and $(\beta_2, \tau_1)$ are cohomologous, and therefore   
$$\Ho^2_N(H,I) = \{[(\beta_1, \tau_1)], [(\beta_2,\tau_2)]\} \cong \mathbb{Z}/2\mathbb{Z}.$$ 

Now we consider the trivial action, that is, $\nu_h = \mu_h = \sigma_h = \Id$ for all $h \in H$ . As above, for computing $2$-cocycles $(\beta, \tau)$ of $H$, it is enough to know their values on $(1,1)$. A detailed computation shows that $(\beta_i, \tau_j)$, $1 \le i, j \le 4$, are all possible $2$-cocycles of $H$, where $\beta_1 = (0, 0)= \tau_1$, $\beta_2 = (0, 1)= \tau_2$, $\beta_3 = (1, 0)= \tau_3$  and $\beta_4 = (1, 1)= \tau_4$.  Let $\theta \in C^1_N$. Notice that $\theta(0) = 0$. Then  $\partial^1(\theta) = (g, f)$, where
$$g(h_1, h_2) = f(h_1, h_2) = \theta(h_2) - \theta(h_1 + h_2) + \theta(h_1) = 0$$
for all $h_1, h_2 \in H$, since $H$ is a trivial brace and the action is trivial. Hence $\B^2_N(H, I)$ is the trivial group, and therefore $\Ho^2_N(H,I)$ is isomorphic to the elementary abelian group of order $16$. 


\section{Cohomology of Braces : A special case}

In this section we generalize the cohomological results of \cite[Sections 2, 3]{LV16} for skew braces.
Let $(H, +, \circ)$ be a left skew brace and $(I, +)$  an abelian group.  Again let   $\nu : (H, \circ) \rightarrow \Aut(I,  +)$  be a left group action and  $\mu: (H, +) \rightarrow \Aut(I,+)$ and $ \sigma : (H, \circ) \rightarrow \Aut(I,+)$ be right group actions. Please note that throughout this section, we'll consider just a triplet of actions $(\nu, \mu, \sigma)$. We'll not need a good triplet except in the situation when the actions are all  trivial, in which case it automatically becomes a good triplet.

Set   $C^n(H;I) := \Fun(H^n, I)$ for $n \ge 1$, where $H^n$ denotes the cartesian product of $n$ copies of $H$. Notice that $C^n(H;I)$ is an abelian group for each $n$. For $n \ge 0$, define the maps $\partial^n : C^n \rightarrow C^{n+1}$ by

\begin{eqnarray*}
(\partial^{n} f)(h_1, \ldots, h_{n+1}) &=&  \nu_{h_1}(f(h_2,\ldots,h_{n+1}))\\
&  & +\;\sum_{i=1}^{n} (-1)^{i}  f(h_1, \ldots, h_i \circ h_{i+1} , \ldots ,  h_{n+1}) \\ 
& & + \; (-1)^{n+1} \ \nu_{h_1\circ h_2 \circ  \cdots  \circ h_{n+1}} (\sigma_{h_{n+1}}(\nu^{-1} _{h_1 \circ h_2 \circ \cdots \circ h_n}(f(h_1, \ldots,h_n)))).
\end{eqnarray*}

Define 
$$RC^n(H; I)=\{f \in Fun(H^n, I) \mid f \;\mbox{is linear in the {\it n}th co-ordinate}\}.$$
Then the restriction of $\partial^{n}$, which we still denote by $\partial^{n}$,  gives the map  
$$\partial^n : RC^n(H; I) \rightarrow RC^{n+1}(H; I)$$
for each $n \ge 1$. Further, let $RC_N^n(H; I)$ denote the set of all $f \in RC^n(H; I)$ which vanish on all degenerate tuples. This becomes a subgroup of $RC^n(H; I)$ and the restriction of $\partial^n$  gives a map from $RC^n_N(H; I) \to RC^{n+1}_N(H; I)$.

\begin{thm}
For $n \ge 1$, $(C^n(H;I) , \partial^{n})$, $(RC^n(H; I), \partial^n)$ and $(RC_N^n(H; I), \partial^n)$ are cochain complexes.
\end{thm}

\begin{proof}
We shall only indicate a proof of the fact that  $(C^n(H;I) , \partial^{n})$ is a cochain complex. Other assertion can be easily  verified from this.  We are required to show that $  \partial^{n+1} \partial^{n} = 0$.  We can write 
$$ \partial^n= \sum_{i=0}^{n+1}(-1)^{i}\partial_{n,i},$$
where 
\begin{eqnarray*}
(\partial_{n, 0}f)(h_1, \ldots,h_{n+1}) &=& \nu_{h_1}(f(h_2,\ldots,h_{n+1})),\\
(\partial_{n,i}f)(h_1, \ldots,h_{n+1}) &=& f(h_1, \ldots,h_i \circ h_{i+1}, \ldots, h_{n+1}) ,  \hspace{.2cm} 1 \leq i \leq n, \\
(\partial_{n,n+1}f)(h_1, \ldots,h_{n+1}) &=& \nu_{h_1\circ h_2 \circ  \cdots  \circ h_{n+1}} (\sigma_{h_{n+1}}(\nu^{-1} _{h_1 \circ h_2 \circ \cdots\circ h_n}(f(h_1, \ldots,h_n)))).
\end{eqnarray*} 

By a direct computation one can easily show that $ \partial_{n,i}  \partial_{n-1,j}=\partial_{n, j+1} \partial_{n-1,i}$ for $0 \le i, j \le n$, which implies $\partial^n  \partial^{n-1} =0$. To demonstrate, we compute in one case. Let $i=0$ and $j \ge 1$ be a fixed integer in the appropriate range. Then
\begin{eqnarray*}
\partial_{n,0}(\partial_{n-1,j}f)(h_1 \ldots h_{n+1}) & = &\nu_{h_1}(\partial_{n-1,j}f)(h_2, \ldots, h_{n+1})\\ 
 &=& \nu_{h_1}(f(h_2, \ldots, h_{j+1} \circ h_{j+2}, \ldots, h_{n+1})).
\end{eqnarray*}
On the other hand
\begin{eqnarray*}
\partial_{n, j+1}(\partial_{n-1,0}f)(h_1, \ldots h_{n+1}) & = & (\partial_{n-1,0}f)(h_1, \ldots, h_{j+1} \circ h_{j+2}, \ldots, h_{n+1})\\
& = &\nu_{h_1}(f(h_2, \ldots, h_{j+1} \circ h_{j+2}, \ldots, h_{n+1})).
\end{eqnarray*}
Hence $\partial_{n,0}\partial_{n-1,j} = \partial_{n, j+1}\partial_{n-1,0}$.  $\hfill$ $\Box$

\end{proof}

Let $\Ho^n(H,I)$,  $\RH^n(H,I)$ and $\RH^n_N(H,I)$, respectively, denote the $n$th cohomology groups of the complexes $(C^n(H;I) , \partial^{n})$, $(RC^n(H; I), \partial^n)$ and $(RC^n_N(H; I), \partial^n)$.  

Let $\STExt_{(\nu, \mu, \sigma)}(H,I)$ be the set of equivalence classes of all extensions of $H$ by $I$ with fixed good triplet of actions $(\nu, \mu, \sigma)$, and  which split as extensions of additive groups. Let $(\beta, \tau) \in \Ho_N^2(H,I)$. If we take $\beta = 0$, then it follows that $\tau \in \RH^2_N(H, I)$. More precisely, $\RH^2_N(H, I) = \{\tau : H \times H \to I  \mid (0, \tau) \in \Ho_N^2(H,I)\}$.  The next result now follows from Theorem \ref{gbij-thm sb} by taking $\beta = 0$.

\begin{thm}
There exists  a one-to-one correspondence between $\RH^2_N(H, I)$ and $\STExt_{(\nu, \mu, \sigma)}(H,I)$.
\end{thm}

Let $(\nu, \mu, \sigma)$ be a good triplet of actions of $H$ on $I$.
Let $\Ho_N^2((H, \circ), I)$ denote the second cohomology group of $(H, \circ)$ with coefficients in the abelian group $I$, where  $(H, \circ)$ acts on $I$ through $\sigma$. For a $2$-cocycle $f$ from $RC^n_N(H; I)$ we can define a map $\bar{f} :  (H, \circ) \times (H, \circ) \to I$ by setting
\begin{equation}\label{eqn-sec4}
\bar{f}(h_1, h_2) = \nu^{-1}_{h_1 \circ h_2}(f(h_1, h_2))
\end{equation}
for all $h_1, h_2 \in (H, \circ)$. It turns out that $\bar{f}$ is  a $2$-cocycle of $(H, \circ)$ with coefficients in $I$. This gives rise to the following

\begin{prop}\label{prop-sec4}
The cohomology group $\RH^2_N(H, I)$ of a skew brace embeds in the cohomology group $\Ho_N^2((H, \circ), I)$ of a group under the association \eqref{eqn-sec4}.
\end{prop}

Let $H$ and $H'$ be  left skew braces, and $I$ and $I'$ be  $H$ and $H'$ modules with $(\nu, \mu, \sigma)$ and  $(\nu', \mu',  \sigma')$, respectively, as  actions.
Let $\alpha:H^\prime \rightarrow H$ and $\zeta: I \rightarrow I^\prime$ be brace homomorphisms.   The pair $(\alpha, \zeta)$ is said to be  \emph{compatible} with the pairs of actions $(\nu, \mu, \sigma)$ and $(\nu^\prime, \mu', \sigma^\prime)$ if the  diagram \eqref{comp-homo} commutes for all actions. Recall that, equivalently saying, $(\alpha, \zeta)$ is  compatible with the actions, if 
$$\zeta(\nu_{\alpha(h^\prime)}(y))=\nu^\prime_{h^\prime}(\zeta(y)),$$
$$\zeta(\mu_{\alpha(h^\prime)}(y))=\mu^\prime_{h^\prime}(\zeta(y))$$
 and 
 $$\zeta(\sigma_{\alpha(h^\prime)}(y))=\sigma^\prime_{h^\prime}(\zeta(y))$$
for all $h' \in H'$ and $y \in I$.
 
Let $(\alpha, \zeta)$ be compatible with the actions $(\nu, \mu, \sigma)$ and $(\nu^\prime, \mu', \sigma^\prime)$.  Let us fix an integer $n \ge 1$.  For $f \in C^n(H, I)$, define $f' : H'^n \to I'$ by
$$f'(h_1', \ldots, h_n') = \zeta\big(f\big(\alpha(h_1), \ldots, \alpha(h_n)\big)\big)$$
for $h_i' \in H'$. Since $(\alpha, \zeta)$ is compatible with the actions, it easily follows that $f' \in C^n(H', I')$.
Define $(\alpha, \zeta)^n : C^n(H, I) \to C^n(H', I')$ by
$$(\alpha, \zeta)^n(f) = f'$$
for all $f \in C^n(H, I)$. It is not difficult to see that $(\alpha, \zeta)^n$ is a homomorphism and  the following diagram commutes
\begin{equation*}
\begin{CD}
 C^n(H, I)   @>{(\alpha, \zeta)^n}>>  C^n(H', I') \\
  @VV{\partial^n}V  @VV{\partial'^n}V\\
 C^{n+1}(H, I)   @>{(\alpha, \zeta)^{n+1}}>>  C^{n+1}(H', I').
\end{CD}
\end{equation*}
Further straightforward computations now give the following
 
 \begin{thm}
For each $n \ge 1$,  the homomorphism $(\alpha, \zeta)^n : C^n(H, I) \to C^n(H', I')$, defined in the preceding para, induces a homomorphism from  $\Ho^n(H, I)$ to $\Ho^n(H', I')$,
 which further induces  homomorphisms from $\RH^n(H, I)$ to $\RH^n(H', I')$ and $\RH^n_N(H, I)$ to $\RH^n_N(H', I')$.
 \end{thm}
 
 Let $K$ be a left ideal of $H$ and $\alpha : K \to H$ be the embedding. For a  triplet of actions $(\nu, \mu, \sigma)$ of $H$ on $I$, $(\nu', \mu', \sigma')$ is a triplet of actions of $K$ on $I$, where $\nu' = \nu|_K$,  $\mu' = \mu|_K$ and $\sigma' = \sigma|_K$. Let $\Id : I \to I$ be the identity homomorphism. Obviously $(\alpha, \Id)$ is compatible with these actions. Hence, as a consequence of the preceding theorem, we have
 
 \begin{cor}\label{cor-sec4}
 There exists a homomorphism from $\RH^n_N(H, I)$ to $\RH^n_N(K, I)$.
 \end{cor}
 Such a homomorphism is called the \emph{restriction} homomorphism, denoted by $\res^H_K$.


 \section{Wells' type exact sequence for skew braces}

In this section we establish various group actions on the set  $\Ext(H, I)$ and construct an exact sequence connecting certain automorphism groups of $E$ with $\Ho_N^2(H, I)$ for a given extension 
$$\mathcal{E} : 0  \rightarrow I \stackrel{i}{\rightarrow} E \stackrel{\pi}{\rightarrow} H \rightarrow 0$$
of left skew braces. We start with an  action of  $\Autb(H) \times \Autb(I)$ on $\Ext(H, I)$.

Let $H$ be any skew left brace and $I$ a trivial brace.  For a pair  $(\phi, \theta) \in \Autb(H) \times \Autb(I)$ of brace automorphisms and an  extension  
$$\mathcal{E} :  0 \rightarrow I \stackrel{i}{\rightarrow} E \stackrel{\pi}{\rightarrow} H \rightarrow 0$$
of $H$ by $I$,  we can define a new extension
$$\mathcal{E}^{(\phi, \theta)} : 0 \rightarrow I \stackrel{i\theta}{\longrightarrow} E \stackrel{\phi^{-1} \pi}{\longrightarrow} H \rightarrow 0$$
of $H$ by $I$. 

Let 
$$\mathcal{E}_1:  0  \rightarrow I \stackrel{i}{\rightarrow} E_1 \stackrel{\pi}{\rightarrow}  H \rightarrow 0$$
 and
  $$\mathcal{E}_2:  0  \rightarrow I \stackrel{i'}{\rightarrow} E_2 \stackrel{\pi'}{\rightarrow}  H \rightarrow 0$$
be two equivalent extensions of $H$ by $I$. Then it is not difficult to show that  the extensions $\mathcal{E}_1^{(\phi, \theta)}$ and $\mathcal{E}_2^{(\phi, \theta)}$ are also equivalent for any $(\phi, \theta) \in \Autb(H) \times \Autb(I)$. Thus, for a given $(\phi, \theta) \in \Autb(H) \times \Autb(I)$, we can define a map from $\Ext(H, I)$ to itself given by 
\begin{equation}\label{act1 sb}
[\mathcal{E}] \mapsto  [ \mathcal{E}^{(\phi, \theta)}].
\end{equation}
 If $\phi$ and $\theta$ are identity automorphisms, than obviously $\mathcal{E}^{(\phi, \theta)} = \mathcal{E}$. It is also easy to see that  
$$[\mathcal{E}]  ^{(\phi_1, \theta_1) (\phi_2, \theta_2)}=  \big([\mathcal{E}]^{(\phi_1, \theta_1)}\big)^{(\phi_2, \theta_2)}.$$
We conclude that the association \eqref{act1 sb} gives an action of the group $\Autb(H) \times \Autb(I)$  on the set $\Ext(H, I)$.

As we know that $\Ext(H, I) = \bigsqcup_{(\nu, \mu, \sigma)} \Ext_{(\nu, \mu, \sigma)}(H, I)$. \emph{Let $(\nu, \mu, \sigma)$ be an arbitrary but fixed good triplet of actions of $H$ on $I$.}  Let $\C_{(\nu, \mu, \sigma)}$ denote the stabiliser of $\Ext_{(\nu, \mu, \sigma)}(H, I)$ in $\Autb(H) \times \Autb(I)$; more explicitly
$$\C_{(\nu, \mu, \sigma)} = \{ (\phi, \theta) \in \Autb(H) \times \Autb(I) \mid \nu_h=\theta^{-1}\nu_{\phi(h)}\theta,  \mu_h = \theta^{-1}\mu_{\phi(h)}\theta \mbox{ and } \sigma_h = \theta^{-1}\sigma_{\phi(h)}\theta \}.$$
Notice that  $\C_{(\nu, \mu, \sigma)}$ is a subgroup of $\Autb(H) \times \Autb(I)$, and it  acts on  $\Ext_{(\nu, \mu, \sigma)}(H, I)$ by the same rule as given in \eqref{act1 sb}. Further notice that $\C_{(\nu, \mu, \sigma)} = \Autb(H) \times \Autb(I)$ when the action $(\nu, \mu, \sigma)$ is trivial.

Next we consider an action of $ \C_{(\nu, \mu, \sigma)}$ on $\Ho^2_N(H, I)$.
Let $(\phi, \theta) \in \Aut(H) \times \Aut(I)$ and $f \in \Fun(H^n, I)$, where $n \ge 1$ be an integer. Define $f^{(\phi, \theta)} : H^n \to I$ by setting
$$f^{(\phi, \theta)}(h_1, h_2, \ldots, h_n) := \theta^{-1}(f(\phi(h_1), \phi(h_2), \ldots, \phi(h_n))\big).$$ 
It is not difficult to see that the group $\Autb(H) \times \Autb(I)$ acts on  the group $\Fun(H^n, I)$ as well as on the group $C_N^{i,j}(H, I)$, by automorphisms, given by the association
\begin{equation}\label{act2 sb}
f \mapsto f^{(\phi, \theta)}.
\end{equation}
It is also obvious that  $\C_{(\nu, \mu, \sigma)}$ acts on both of these sets. We are interested in the action of $\C_{(\nu, \mu, \sigma)}$ on $\Ho_N^2(H, I)$. The association \eqref{act2 sb} induces an action of $\C_{(\nu, \mu, \sigma)}$ on $C^2_N = C^{0,2}_N(H, I) \oplus C^{1,1}_N(H, I)$ by setting 
\begin{equation}\label{act3 sb}
(\beta, \tau) \mapsto \big(\beta^{(\phi, \theta)}, \tau^{(\phi, \theta)}\big).
\end{equation}

\begin{lemma}\label{lemma-act3 sb}
For $(\phi, \theta) \in \C_{(\nu, \mu, \sigma)}$,  the following hold:

 (i) If $(\beta, \tau) \in \Z^2_N(H,I)$, then $\big(\beta^{(\phi, \theta)}, \tau^{(\phi, \theta)}\big) \in \Z^2_N(H,I)$.

(ii) If $(\beta, \tau) \in \B^2_N(H,I)$, then $\big(\beta^{(\phi, \theta)}, \tau^{(\phi, \theta)}\big) \in \B^2_N(H,I)$.

Hence the association \eqref{act3 sb} gives an action of $\C_{(\nu, \mu, \sigma)}$ on $\Ho_N^2(H, I)$ by automorphisms if we define
$$[(\beta, \tau)]^{(\phi, \theta)} =[\big(\beta^{(\phi, \theta)}, \tau^{(\phi, \theta)}\big)].$$
\end{lemma}
\begin{proof} 
That $\partial^{0,2}_v \big(\beta^{(\phi, \theta)}\big)$ is a zero function in $\C_N^3$ holds trivially.  Since $(\phi, \theta) \in \C_{(\nu, \mu, \sigma)}$, we have
$$\nu_{h_1}(\theta^{-1}(\tau(\phi(h_2), \phi(h_3))))=\theta^{-1}\big(\nu_{\phi(h_1)}(\tau(\phi(h_2), \phi(h_3)))\big)$$
and 
\begin{equation*}
 \nu_{h_1 \circ h_2}(\sigma_{h_2}(\nu^{-1}_{h_1}(\theta^{-1}(\phi(h_1), \phi(h_2)))))=\theta^{-1}\big(\nu_{\phi(h_1 \circ h_2)}(\sigma_{\phi(h_2)}(\nu^{-1}_{\phi(h_1)}(\tau(\phi(h_1),\phi(h_2)))))\big).
  \end{equation*}
Using these identities, for all $h_1, h_2, h_3 \in H$, it follows that
 \begin{equation*}
 \partial^{1,1}_h \big(\tau^{(\phi, \theta)}\big)(h_1, h_2, h_3) = \theta^{-1}\big( \partial^{1,1}_h (\tau)(\phi(h_1), \phi(h_2), \phi(h_3))\big) = 0.
 \end{equation*}
 Similar computations show that 
 \begin{equation*}
  \big(\partial^{0,2}_h\big(\beta^{(\phi, \theta)}\big) - \partial^{1,1}_v\big(\tau^{(\phi, \theta)}\big)\big)(h_1, h_2, h_3) = \theta^{-1}\big((\partial^{0,2}_h (\beta) - \partial^{1,1}_v(\tau)) (\phi(h_1), \phi(h_2), \phi(h_3))\big) = 0.
 \end{equation*}
This proves assertion (i).
 
 Now assume that  $(\beta, \tau) \in \B^2_N(H,I)$.  Then there exists a $\lambda \in C^1_N$ such that  $(\beta, \tau)=\partial^{1}(\lambda)$. It is now not difficult to see that 
 $$\big(\beta^{(\phi, \theta)}, \tau^{(\phi, \theta)}\big) = \partial^{1}\big(\lambda^{(\phi, \theta)}\big),$$
which establishes assertion (ii). That the association under consideration gives an action of $\C_{(\nu, \mu,  \sigma)}$ on $\Ho_N^2(H, I)$ by automorphisms, is now straighforward, which completes the proof. \hfill $\Box$

 \end{proof}

\begin{remark}
The action of $\C_{(\nu, \mu, \sigma)}$  on $\Ho^2_N(H,I)$, defined in the preceding lemma, can be transferred  on  
$\Ext_{(\mu,\sigma, \nu)}(H,I)$ through the bijection given in  Theorem  \ref{gbij-thm sb}. Notice that the resulting action of $\C_{(\nu, \mu, \sigma)}$ on $\Ext_{(\nu, \mu, \sigma)}(H,I)$ agrees with the action defined by \eqref{act1 sb}.
\end{remark}

We now consider the action of  $\Ho^2_N(H,I)$ onto itself by right translation, which is faithful and transitive. Again using Theorem \ref{gbij-thm sb}, we can transfer this action on $\Ext_{(\nu, \mu, \sigma)}(H,I) = \{[(H, I, \nu, \mu, \sigma, \beta, \tau)] \mid [(\beta, \tau)] \in \Ho^2_N(H,I)\}$. More precisely, for $[(\beta_1, \tau_1)] \in \Ho^2_N(H,I)$, the action is given by
$$[(H, I, \nu, \mu, \sigma, \beta, \tau)]^{[(\beta_1, \tau_1)]} = [(H, I, \nu, \mu, \sigma, \beta + \beta_1, \tau + \tau_1)]$$
for all  $[(H, I, \nu, mu, \sigma, \beta, \tau)] \in \Ext_{(\nu, \mu, \sigma)}(H,I)$. Notice that this action is also faithful and transitive.

Consider the semidirect product $\Gamma :=\C_{(\nu, \mu, \sigma)} \ltimes \Ho^2_N(H,I)$  under  the action defined in Lemma \ref{lemma-act3 sb}. We wish to define an action of $\Gamma$ on $\Ext_{(\nu, \mu, \sigma)}(H,I)$.  For $(c, h) \in \Gamma$ and $[\mathcal{E}] \in \Ext_{(\nu, \mu, \sigma)}(H,I)$, define 
\begin{equation}\label{act4 sb}
[\mathcal{E}]^{(c, h)} = ([\mathcal{E}]^c)^h.
\end{equation}

\begin{lemma}\label{wells2 sb}
The rule  in \eqref{act4 sb}   gives an action of  $\Gamma$ on $\Ext_{(\nu, \mu, \sigma)}(H,I)$.
\end{lemma}
\begin{proof}
 Notice that for $(c_1,h_1),  (c_2, h_2) \in \Gamma$,  $(c_1,h_1)(c_2,h_2)=(c_1c_2, h_1^{c_2} \, h_2)$.  So, it is enough to show that $\big([\mathcal{E}]^h\big)^c = \big([\mathcal{E}]^c\big)^{h^c}$ for each $c \in \C_{(\nu, \mu, \sigma)}$, $h \in \Ho^2_N(H,I)$ and $[\mathcal{E}] \in \Ext_{(\nu, \mu, \sigma)}(H,I)$.
We know that  $[\mathcal{E}]=[(H,I,\nu,\sigma,\beta,\tau)]$ for some $[(\beta, \tau)] \in \Ho^2_N(H,I)$.  Then, for $h=[(\beta_{h}, \tau_h)] \in \Ho^2_N(H, I)$, we have
\begin{eqnarray*}
\big([\mathcal{E}]^{h}\big)^c &=& [(H,I,\nu, \mu, \sigma, (\beta+\beta_h)^c, (\tau+\tau_h)^c)]\\
&=& [(H,I,\nu, \mu, \sigma, \beta^c+ \beta_h^c, \tau^c+ \tau_h^c)]\\
&=& ([(H,I,\nu, \mu, \sigma, \beta^c, \tau^c)])^{h^c}\\
&=& \big([\mathcal{E}]^c\big)^{h^c}.
\end{eqnarray*}
The proof is now complete. \hfill $\Box$

\end{proof}

 Let $[\mathcal{E}] \in \Ext_{(\nu, \mu, \sigma)}(H,I)$ be a fixed extension. Since the action of $\Ho^2_N(H,I)$ on $\Ext_{(\nu, \mu, \sigma)}(H,I)$ is transitive and faithful, for  each $c \in \C_{(\nu, \mu, \sigma)}$, there exists a unique element (say) $h_c$  in  $\Ho^2_N(H,I)$ such that  
 $$[\mathcal{E}]^{c} = [\mathcal{E}]^{h_c}.$$
 We thus have a well defined map $ \omega(\mathcal{E}): \C_{(\nu, \mu, \sigma)} \rightarrow \Ho^2_N(H,I)$ given by
 \begin{equation}\label{wells-map sb}
 \omega(\mathcal{E})(c)=h_c
 \end{equation}
 for $c \in \C_{(\nu, \mu, \sigma)}$. 
  
\begin{lemma}\label{wells3 sb}
The map $ \omega(\mathcal{E}): \C_{(\nu, \mu, \sigma)} \rightarrow \Ho^2_N(H,I)$ given in \eqref{wells-map sb} is a derivation with respect to the action of $\C_{(\nu, \mu, \sigma)}$ on $H^2_N(H,I)$ given in \eqref{act3 sb}.
\end{lemma}
\begin{proof}
Let $c_1, c_2 \in \C_{(\nu, \mu,  \sigma)}$ and  $\omega(\mathcal{E})(c_1c_2) = h_{c_1c_2}$.  Thus, by the definition of $\omega(\mathcal{E})$,   $[\mathcal{E}]^{c_1c_2} = [\mathcal{E}]^{h_{c_1c_2}}$.  Using the fact that  $\big([\mathcal{E}]^{h}\big)^{c} = \big([\mathcal{E}]^{c}\big)^{h^c}$ for each $c \in \C_{(\nu, \mu, \sigma)}$, $h \in \Ho^2_N(H,I)$, we have
\begin{eqnarray*}
 [\mathcal{E}]^{h_{c_1c_2}} & = &[\mathcal{E}]^{(c_1c_2)}\\
 &=& \big([\mathcal{E}]^{c_1}\big)^{c_2}\\
 &=& \big([\mathcal{E}]^{h_{c_1}}\big)^{c_2}\\
 &= & \big([\mathcal{E}]^{c_2}\big)^{(h_{c_1})^{c_2}}\\
 &=& \big([\mathcal{E}]^{h_{c_2}}\big)^{(h_{c_1})^{c_2}}\\
 &=& [\mathcal{E}]^{\big(h_{c_2} + (h_{c_1})^{c_2} \big)}.
 \end{eqnarray*}
 Since the action of $\Ho^2_N(H,I)$ on $\Ext_{(\nu, \mu, \sigma)}(H,I)$  is faithful,  it follows that $h_{c_1c_2} = (h_{c_1})^{c_2} + h_{c_2}$. This implies that $\omega(\mathcal{E})(c_1c_2)=\big(\omega(\mathcal{E})(c_1)\big)^{c_2}+\omega(\mathcal{E})(c_2)$; hence $\omega(\mathcal{E})$ is a derivation. \hfill $\Box$
 
\end{proof}

Let 
$$\mathcal{E}: 0 \rightarrow I \rightarrow E \overset{\pi}\rightarrow H $$
be an extension of a left skew brace $H$ by a trivial brace $I$ such that $[\mathcal{E}] \in \Ext_{(\nu, \mu, \sigma)}(H,I)$.
Let  $\Autb_I(E)$ denote the subgroup of $\Autb(E)$ consisting of all automorphisms of $E$ which normalize $I$, that is,
$$\Autb_I(E) := \{ \gamma \in \Autb(E) \mid \gamma(y) \in I \mbox{ for all }  y \in I\}.$$ 
For $\gamma \in \Autb_I(E)$, set $\gamma_I := \gamma |_I$, the restriction of $\gamma$ to $I$, and $\gamma_H$ to be the automorphism of $H$ induced by $\gamma$. More precisely, $\gamma_H(h) = \pi(\gamma(s(h)))$ for all $h \in H$, where $s$ is an st-section of $\pi$. Notice that the definition of $\gamma_H$ is independent of  the choice of an st-section. Define a map $\rho(\mathcal{E}) :  \Autb_I(E) \rightarrow  \Autb(H) \times \Autb(I)$ by
$$\rho(\mathcal{E})(\gamma)=(\gamma_H, \gamma_I).$$ 
Although $\omega(\mathcal{E})$ is not a homomorphism, but we can still talk about its set theoretic kernel, that is,
$$\Ker(\omega(\mathcal{E})) = \{c \in C_{(\mu,\sigma, \nu)} \mid [\mathcal{E}]^c=[\mathcal{E}]\}.$$

\begin{prop}\label{wells4 sb}
For the extension $\mathcal{E}$,   $\IM(\rho(\mathcal{E})) \subseteq \C_{(\nu, \mu, \sigma)}$ and  
$\IM(\rho(\mathcal{E})) = \Ker(\omega(\mathcal{E}))$.
\end{prop}
 \begin{proof}
For the first assertion, we are required  to show that $\nu_{h} = \gamma_I^{-1} \nu_{\gamma_H(h)} \gamma_I$, $\mu_{h} = \gamma_I^{-1} \mu_{\gamma_H(h)} \gamma_I$ and  $\sigma_{h} = \gamma_I^{-1} \sigma_{\gamma_H(h)} \gamma_I$ for all $h \in H$. Let $s$ be an st-section of $\pi$  and $ x \in E$. Notice that $\gamma_I^{-1}$ is the restriction of $\gamma^{-1}$ on $I$.   Also notice that  for a given $x \in E$, $s(\pi(x)) = x \circ y_h$ for some $y_h \in I$.  Now for $h \in H$ and $ y \in I$, we have
\begin{eqnarray*}
\gamma_I^{-1} \nu_{\gamma_H(h)} \gamma_I(y) &=& \gamma^{-1} \big(\nu_{\pi(\gamma(s(h)))}(\gamma(y))\big)\\
&=& \gamma^{-1}\big(-s(\pi(\gamma(s(h)))) + s(\pi(\gamma(s(h)))) \circ \gamma (y)\big)\\
&=& \gamma^{-1}\big(- \gamma(s(h)) \circ y_{\gamma(s(h))} + \gamma(s(h)) \circ y_{\gamma(s(h))}\circ \gamma(y) \big)\\
&=& - s(h) \circ \gamma^{-1}( y_{\gamma(s(h))}) + s(h) \circ \gamma^{-1}( y_{\gamma(s(h))})\circ y \\
&=& - s(h) \circ \gamma^{-1}( y_{\gamma(s(h))}) + s(h)\circ \big( \gamma^{-1}( y_{\gamma (s(h))}) + y\big)\\
&=& -s(h) + s(h)\circ y \quad \mbox{ (using \eqref{bcomp})}\\
&=& \nu_h(y).
\end{eqnarray*}
Hence $\nu_{h} = \gamma_I^{-1} \nu_{\gamma_H(h)} \gamma_I$.  One can similarly show that $\mu_{h} = \gamma_I^{-1} \mu_{\gamma_H(h)} \gamma_I$ and $\sigma_{h} = \gamma_I^{-1} \sigma_{\gamma_H(h)} \gamma_I$. 

Now we prove the second assertion. Let $\rho(\mathcal{E})(\gamma) =  (\gamma_H, \gamma_I)$ for  $\gamma \in \Autb_I(E)$.  We know that $s(\pi(x))=x+y_x $ for some $y_x \in I$.  Thus we have
\begin{eqnarray*}
\gamma_H^{-1}(\pi(\gamma(s(h)))) &=& \pi\big(\gamma^{-1}(s(\pi(\gamma(s(h)))))\big)\\
&=& \pi\big(\gamma^{-1}\big(\gamma(s(h))+y_{\gamma(s(h))}\big)\big)\\
&= & h,
\end{eqnarray*}
 which implies that the diagram  
$$\begin{CD}
0 @>>> I @>>> E@>{{\pi} }>> H @>>> 0\\ 
&& @V{\text{Id}} VV @V{\gamma} VV @V{\text{Id} }VV \\
0 @>>> I @>{\gamma_I}>> E @>{\gamma_H^{-1} \pi}>> H  @>>> 0
\end{CD}$$
commutes. Hence $[(\mathcal{E})]^{(\gamma_H, \gamma_I)} =[(\mathcal{E})]$, which shows that $\IM(\rho(\mathcal{E})) \subseteq \Ker(\omega(\mathcal{E}))$. 

Conversely, if $(\phi, \theta) \in  \Ker(\omega(\mathcal{E}))$,  then there exists a brace homomorphism $\gamma: E \rightarrow E$ such that the diagram 
$$\begin{CD}
0 @>>> I @>>> E@>{{\pi} }>> H @>>> 0\\ 
&& @V{\text{Id}} VV@V{\gamma} VV @V{\text{Id} }VV \\
0 @>>> I @>{\theta}>> E @>{ \phi^{-1} \pi}>> H  @>>> 0
\end{CD}$$
commutes. It is now obvious that $\gamma \in \Autb_I(E)$, $\theta = \gamma_I$ and $\phi = \gamma_H$. Hence
 $\rho(\mathcal{E})(\gamma)=(\phi, \theta)$, which completes the proof.  \hfill   $\Box$

\end{proof}

Continuing with the above setting, set $\Autb^{H, I}(E) := \{\gamma \in \Autb_I(E) \mid \gamma_I = \Id,  \gamma_H = \Id\}$. Notice that  $\Autb^{H,I}(E)$ is precisely the kernel of $\rho(\mathcal{E})$. Hence, using Proposition \ref{wells4 sb}, we get

\begin{thm}\label{wells5 sb}
Let $\mathcal{E}: 0 \rightarrow I \rightarrow E \overset{\pi}\rightarrow H$ be a extension of a left skew  brace  $H$ by a trivial  brace $I$ such that $[\mathcal{E}] \in \Ext_{(\nu, \mu, \sigma)}(H,I)$. Then we have the following exact sequence of groups 
$$0 \rightarrow \Autb^{H,I}(E) \rightarrow \Autb_I(E) \stackrel{\rho(\mathcal{E})}{\longrightarrow} \C_{(\nu, \mu,\sigma)} \stackrel{\omega(\mathcal{E})}{\longrightarrow} \Ho^2_N(H,I),$$
where $\omega(\mathcal{E})$ is, in general,  only a derivation.
\end{thm}

We further prove
\begin{prop}\label{wells6 sb}
Let  $\mathcal{E} :  0 \rightarrow I \rightarrow E \overset{\pi}\rightarrow H$ be an extension of $H$ by $I$ such that $[\mathcal{E}] \in \Ext_{(\nu, \mu, \sigma)}(H, I)$.  Then  $\Autb^{H,I}(E) \cong \Z^1_N(H,I)$.
\end{prop}
 \begin{proof}
We know that every element $x \in E$ has a unique expression of the form $x=s(h)+y=s(h) \circ \nu^{-1}_h(y)$ for some $h \in H$ and $y \in I$. Let us define a map $\eta: \Z^1_N(H,I)  \rightarrow  \Autb^{H,I}(E)$ by 
$$\eta(\lambda)((s(h)+y) = s(h)+ \lambda(h) + y,$$
where $\lambda \in \Z^1_N(H,I)$.
Notice that the image of $\eta(\lambda)$ is independent of the choice of an st-section.
We claim that $\eta(\lambda) \in  \Autb^{H,I}(E)$.  For $h_1,h_2 \in H$ and $y_1, y_2 \in I$, we know that 
$$\tau(h_1,h_2) =\nu_{h_1\circ h_2}\big(s(h_1 \circ h_2)^{-1}\circ s(h_1)\circ s(h_2)\big)$$
 as defined in \eqref{cocycle2 sb}. Set $\tau_1 = \nu_{h_1 \circ h_2}^{-1} \tau$. Also notice that $s(h) \circ y = s(h) + \nu_h(y)$. We then have 
\begin{eqnarray*}
\eta(\lambda)\big((s(h_1)+y_1) \circ(s(h_2) + y_2)\big) &= &\eta(\lambda)\big((s(h_1)\circ \nu^{-1}_{h_1}(y_1)) \circ (s(h_2)\circ \nu^{-1}_{h_2}(y_2))\big)\\
&= &\eta(\lambda)\big(s(h_1) \circ s(h_2)\circ \sigma_{h_2}(\nu^{-1}_{h_1}(y_1))\circ \nu^{-1}_{h_2}(y_2)\big)\\
&=& \eta(\lambda)\big(s(h_1\circ h_2) \circ \big(\tau_1(h_1,h_2)+ \sigma_{h_2}(\nu^{-1}_{h_1}(y_1))+ \nu^{-1}_{h_2}(y_2)\big)\big)\\
&=& \eta(\lambda)\big(s(h_1\circ h_2)+\nu_{h_1\circ h_2} \big(\tau_1(h_1,h_2)+ \sigma_{h_2}(\nu^{-1}_{h_1}(y_1))+ \nu^{-1}_{h_2}(y_2)\big)\big)\\
&=& s(h_1\circ h_2)+\lambda(h_1 \circ h_2)+ \tau(h_1,h_2)+ \nu_{h_1\circ h_2} (\sigma_{h_2}(\nu^{-1}_{h_1}(y_1)))\\
& & + \nu_{h_1}(y_2)\\
&=& s(h_1\circ h_2) +\tau(h_1,h_2) +\nu_{h_1\circ h_2}(\sigma_{h_2}(\nu^{-1}_{h_1}(y_1+\lambda(h_1))))\\
& & +\nu_{h_1}(y_2+\lambda(h_2))\\
&=& (s(h_1)+\lambda(h_1)+y_1)\circ (s(h_2)+\lambda(h_2)+y_2)\\
&=& \eta(\lambda)(s(h_1)+y_1) \circ \eta(\lambda)(s(h_2) +y_2).
\end{eqnarray*}
It is easy to see that 
$$ \eta(\lambda)\big((s(h_1)+y_1) +(s(h_2) +y_2)\big) = \eta(\lambda)(s(h_1)+y_1)+\eta(\lambda)(s(h_2)+y_2).$$
That $\eta$ is injective, follows from the injectivity of $s$ and the fact that $\lambda(0) = 0$.  Surjectivity  of $\eta(\lambda)$ is immediate. This shows that $\eta(\lambda)$ is an automorphism of $E$. We also  have $\eta(\lambda)(s(h))=s(h)+\lambda(h)$ and $\eta(\lambda)(y)=y$ for all $h \in H$ and $ y \in I$. Hence  $\eta(\lambda) \in  \Autb^{H,I}(E)$. 

 We'll now define a map $\zeta :  \Autb^{H,I}(E)  \rightarrow \Z^1_N(H,I)$. For $\gamma \in \Autb^{H,I}(E)$ and  $h \in H$, there exists a unique element (say) $y^{\gamma}_h \in I$ such that  $\gamma(s(h)) = s(h)+ y^{\gamma}_h$ for some unique $y^{\gamma}_h \in I$. Thus, for $h \in H$,  define $\zeta$ by 
 $$\zeta(\gamma)(h) = y^{\gamma}_h.$$
 Notice that this is independent of the choice of an st-section. Since $\sigma_h^{-1} = \sigma_{h^{-1}}$, it follows that 
 $$\gamma(s(h)^{-1})=s(h)^{-1} \circ \sigma_{h^{-1}}\big(\nu^{-1}_h((y^{\gamma}_h)^{-1})\big)$$
for all $h \in H$. Further, since $s(h_1 \circ h_2)^{-1} \circ s(h_1) \circ s(h_2) \in I$ for all $h_1, h_2 \in H$, we have
\begin{eqnarray*}
s(h_1 \circ h_2)^{-1} \circ s(h_1) \circ s(h_2)&=&\gamma\big(s(h_1 \circ h_2)^{-1} \circ s(h_1) \circ s(h_2)\big)\\
&=& \gamma(s(h_1 \circ h_2)^{-1}) \circ \gamma(s(h_1)) \circ \gamma(s(h_2))\\
&=& s(h_1 \circ h_2)^{-1}\circ \sigma_{(h_1\circ h_2)^{-1}}\big(\nu^{-1}_{h_1\circ h_2}(y^\gamma_{h_1\circ h_2})^{-1}\big) \circ s(h_1) \circ \nu^{-1}_{h_1}(y^{\gamma}_{h_1})\\
& & \circ \, s(h_2) \circ \nu^{-1}_{h_2}(y^{\gamma}_{h_2})\\
&=&s(h_1 \circ h_2)^{-1}\circ \sigma_{(h_1\circ h_2)^{-1}}\big(\nu^{-1}_{h_1\circ h_2}(y^\gamma_{h_1\circ h_2})^{-1}\big) \circ \sigma_{h_1^{-1}}(\nu^{-1}_{h_1}(y^{\gamma}_{h_1}))\\
& & \circ \; \sigma_{(h_1 \circ h_2)^{-1}}\big( \nu^{-1}_{h_2}(y^{\gamma}_{h_2})\big)\circ s(h_1) \circ s(h_2),
 \end{eqnarray*}
which,  using $(y^{\gamma}_{h_1\circ h_2})^{-1}=-y^{\gamma}_{h_1\circ h_2}$, implies that 
$$\sigma_{(h_1\circ h_2)^{-1}}\big(\nu^{-1}_{h_1\circ h_2}(y^{\gamma}_{h_1\circ h_2})\big) = \sigma_{h_1^{-1}}(\nu^{-1}_{h_1}(y^{\gamma}_{h_1})) +  \sigma_{(h_1 \circ h_2)^{-1}}\big( \nu^{-1}_{h_2}(y^{\gamma}_{h_2})\big).$$
This, on further simplification, finally gives
$$y^{\gamma}_{h_1 \circ h_2} = \nu_{h_1\circ h_2}\big(\sigma_{h_2} (\nu^{-1}_{h_1}(y^{\gamma}_{h_1}))\big)+ \nu_{h_1}(y^{\gamma}_{h_2}).$$
For $h_1, h_2 \in H$, it easily follows that
$$y_{h_1 + h_2}^{\gamma} = \mu_{h_2}(y_{h_1}^{\gamma}) + y_{h_2}^{\gamma}.$$
We have now proved that  $\zeta(\gamma)$ is a derivation. That both $\eta$ and $\zeta$ are homomorphisms, and $\eta \zeta$ and $\zeta \eta$ are, respectively, the identity elements of  $\Autb^{H,I}(E)$ and $\Z^1_N(H,I)$ is obvious.  Hence $\Autb^{H,I}(E) \cong  \Z^1_N(H,I)$, and the proof is complete. \hfill $\Box$

 \end{proof}

We finally get the following Wells' like exact sequence for skew braces.
\begin{thm}\label{wells7 sb}
Let $\mathcal{E}: 0 \rightarrow I \rightarrow E \overset{\pi}\rightarrow H$ be an extension of a left skew brace  $H$ by a trivial  brace $I$ such that $[\mathcal{E}] \in \Ext_{(\nu, \mu, \sigma)}(H,I)$. Then we have the following exact sequence of groups 
$$0 \rightarrow \Z^1_N(H,I) \rightarrow \Autb_I(E) \stackrel{\rho(\mathcal{E})}{\longrightarrow} \C_{(\nu, \mu,\sigma)} \stackrel{\omega(\mathcal{E})}{\longrightarrow} \Ho^2_N(H,I),$$
where $\omega(\mathcal{E})$ is, in general,  only a derivation.
\end{thm}

A pair of brace automorphisms $(\phi, \theta) \in \Autb(H) \times \Autb(I)$ is said to be \emph{inducible} if $(\phi, \theta) \in \im(\rho(\mathcal{E}))$. As a straightforward application of Theorem \ref{wells7 sb}, we get

\begin{cor}
Let $\mathcal{E}: 0 \rightarrow I \rightarrow E \overset{\pi}\rightarrow H$ be an extension of a left skew brace  $H$ by a trivial  brace $I$ such that $[\mathcal{E}] \in \Ext_{(\nu, \mu, \sigma)}(H,I)$. If  $\Ho^2_N(H,I)$ is a trivial group, then  every element of $\C_{(\nu, \mu,\sigma)}$ is inducible. 
\end{cor}

Let $\mathcal{E}: 0 \rightarrow I \rightarrow E \rightarrow H \rightarrow 0$ be an extension of a left skew brace $H$ by a trivial brace $I$. Then $I$ can be viewed as an $H$-module through the corresponding good triplet of actions $(\nu, \mu, \sigma)$ as defined above.  An automorphism $\phi$ of the skew brace  $H$ defines a new $H$-module structure on $I$ given by $(\nu \phi, \mu \phi ,\sigma \phi)$, which we denote by $I_{\phi}$, where $\nu \phi(h) = \nu_{\phi(h)}$,  $\mu \phi(h) (y)  = - s(\phi(h)) + y + s(\phi(h))$ and  $\sigma \phi(h)(y) = s(\phi(h))^{-1}ys(\phi(h))$ for $h \in H$ and $y \in I$. It is not difficult to show that the automorphism $\phi$ induces an isomorphism $\phi^*$ of cohomology groups
$\phi^* : \Ho^2_N(H,I) \rightarrow \Ho^2_N(H,I_\phi)$
defined by
$$\phi^*([(\beta, \tau)])= [\big(\beta^{(\phi, \Id)}, \tau^{(\phi, \Id)}\big)].$$
 Further, any $H$-module isomorphism  $\theta : I \rightarrow I_\phi$ induces an isomorphism $\theta^*$ of cohomology groups
 $\theta^* : \Ho^2_N(H,I) \rightarrow \Ho^2_N(H,I_\phi)$
 given by
 $$\theta^*([(\beta, \tau)]) = [\big(\beta^{(\Id, \theta^{-1})}, \tau^{(\Id, \theta^{-1})}\big)].$$
 With this set-up, we have
 
 \begin{thm} 
A pair of automorphisms $(\phi, \theta) \in \Autb(H) \times \Autb(I)$  is inducible if and only if the following conditions hold:

(i)  $\theta : I \rightarrow I_\phi $ is an isomorphism of $H$-modules.

(ii)  $\theta^*([(\beta, \tau)])= \phi^*([(\beta, \tau)])$.
\end{thm}
\begin{proof}
Let  $(\phi, \theta)$ be inducible. Then there exists an automorphism $\gamma \in \Autb_I(E)$ such that $(\phi, \theta) = (\gamma_H, \gamma_I)$.   For $h \in H $ and $y \in I$, we know that $\gamma(s(h)) = s(\gamma_H(h)) \circ y_h$ for some $y_h \in I$. We have
\begin{eqnarray}\label{modeqn1sb}
\gamma_I(\nu_h(y))& =&\gamma\big(-s(h)+s(h)\circ y\big)\nonumber\\
&=& -\gamma(s(h))+\gamma(s(h)) \circ \gamma(y) \nonumber\\
&=&-(s(\gamma_H(h))\circ y_h)+ s(\gamma_H(h))\circ( y_h + \gamma(y)) \nonumber\\
&=&- s(\gamma_H(h))+ s(\gamma_H(h))\circ  \gamma_I(y) \ \ \ \mbox{(using \eqref{bcomp})}\nonumber\\
&=& \nu_{\gamma_H(h)}(\gamma_I(y)).
\end{eqnarray}

Similarly it also follows that 
\begin{equation}\label{modeqn1sba }
\gamma_I(\sigma_h(y))=\sigma_{\gamma_H(h)}(\gamma_I(y))
\end{equation}
and 
\begin{equation}\label{modeqn1sbb}
\gamma_I(\mu_h(y))=\mu_{\gamma_H(h)}(\gamma_I(y)).
\end{equation}
 This shows that  $(\Id, \gamma_I)$ is compatible with the pairs of actions $(\nu, \mu, \sigma)$ and $(\nu \gamma_H, \mu \gamma_H,\sigma \gamma_H)$; hence condition (i) holds. 
 
 For each $h \in H$, there exists a unique element (say) $\lambda(h)$ in $I$ such that  $ \gamma(s(h)) =s(\gamma_H(h)) + \lambda(h)$.  It turns out that $\lambda : H \rightarrow I$ given by $h \mapsto \lambda(h)$  lies in $C_N^1$. Given elements $x_1, x_2 \in E$ have unique expressions of the form $x_1= s(h_1) + y_1$ and $x_2 = s(h_2) + y_2$ for some $h_1, h_2 \in H$ and $y_1$ and $ y_2 \in I$.  Now
\begin{eqnarray*}
\gamma(x_1+x_2) &=& \gamma\big(s(h_1+ h_2)+\beta(h_1,h_2)+\mu_{h_2}(y_1)+y_2\big)\\
&=& \gamma(s(h_1+ h_2))+\gamma_I(\beta(h_1,h_2))+\gamma_I (\mu_{h_2}(y_1)+y_2)\\
&=& s\big(\gamma_H(h_1)+\gamma_H(h_2)\big)+\lambda(h_1+h_2)+\gamma_I(\beta(h_1,h_2))+\gamma_I (\mu_{h_2}(y_1)+y_2).
\end{eqnarray*}
On the other hand, we have
\begin{eqnarray} 
\gamma(x_1)+\gamma(x_2)&=&s(\gamma_H(h_1))+\lambda(h_1)+\gamma_I(y_1)+s(\gamma_H(h_2))+\lambda(h_2)+\gamma_I(y_2) \nonumber\\
&=& s(\gamma_H(h_1))+s(\gamma_H(h_2))+\mu_{\gamma_H(h_2)}(\lambda(h_1)) +\mu_{\gamma_H(h_2)}(\gamma_I(y_1))+ \lambda(h_2) + \gamma_I(y_2) \nonumber\\
&=& s(\gamma_H(h_1)+\gamma_H(h_2))+\beta(\gamma_H(h_1), \gamma_H(h_2))+\mu_{\gamma_H(h_2)}(\lambda(h_1)) \nonumber\\
&& + \; \mu_{\gamma_H(h_2)}(\gamma_I(y_1))+ \lambda(h_2) + \gamma_I(y_2).\nonumber
\end{eqnarray}
Since $\gamma(x_1+x_2) = \gamma(x_1)+\gamma(x_2)$, preceding two equations, using \eqref{modeqn1sbb}, give
\begin{equation}\label{modeqn2}
\gamma_I(\beta(h_1, h_2))-\beta(\gamma_H(h_1), \gamma_H(h_2))=\mu_{\gamma_H(h_2)}(\lambda(h_1))-\lambda(h_1+h_2)+\lambda(h_2).\end{equation}

Notice that  $s(h) +\lambda(h)=s(h)\circ \nu^{-1}_{h}(\lambda(h))$ and $s(h) \circ \lambda(h) = s(h) + \nu_{h}(\lambda(h))$. By this and  using  \eqref{modeqn1sb}, \eqref{modeqn1sba } and \eqref{modeqn1sbb} we  have
\begin{eqnarray*}
\gamma(x_1 \circ x_2) &=&  \gamma\big((s(h_1)+y_1) \circ ((s(h_2)+y_2)\big)\\
 &=&  \gamma\big(s(h_1 \circ h_2) \circ \big(\nu^{-1}_{h_1 \circ h_2}(\tau(h_1, h_2)) + \sigma_{h_2}(\nu^{-1}_{h_1}(y_1)) + \nu^{-1}_{h_2}(y_2)\big)\big)\\
 &=& \gamma\big(s(h_1 \circ h_2)+\nu_{h_1
\circ h_2}(\sigma_{h_2}(\nu^{-1}_{h_1}(y_1)))+\nu_{h_1}(y_2)+\tau(h_1,h_2)\big)\\
&=& s(\gamma_H(h_1 \circ h_2)) +\lambda(h_1 \circ h_2)+\gamma_I(\nu_{h_1
\circ h_2}(\sigma_{h_2}(\nu^{-1}_{h_1}(y_1))))+\gamma_I(\nu_{h_1}(y_2))+\gamma_I(\tau(h_1, h_2))\\
&=& s(\gamma_H(h_1 \circ h_2)) + \lambda(h_1 \circ h_2)+\nu_{\gamma_H(h_1
\circ h_2)}(\sigma_{\gamma_H(h_2)}(\nu^{-1}_{\gamma_H(h_1)}(\gamma_I(y_1))))\\
& & + \; \nu_{\gamma_H(h_1)}(\gamma_I(y_2))+\gamma_I(\tau(h_1, h_2)).
\end{eqnarray*} 
By similar computations, on the other hand, we get
\begin{eqnarray*}
 \gamma(x_1) \circ \gamma(x_2) &=&  s(\gamma_H(h_1 \circ h_2))+\nu_{\gamma_H(h_1
\circ h_2)}(\sigma_{\gamma_H(h_2)}(\nu^{-1}_{\gamma_H(h_1)}(\gamma_I(y_1))))\\
&& + \; \nu_{\gamma_H(h_1 \circ h_2)}(\sigma_{\gamma_H(h_2)}(\nu^{-1}_{\gamma_H(h_1)}(\lambda(h_1))))
  +\nu_{\gamma_H(h_1)}(\gamma_I(y_2))\\
  && + \; \nu_{\gamma_H(h_1)}(\lambda(h_2))+\tau(\gamma_H(h_1), \gamma_H(h_2)).
\end{eqnarray*}
Preceding two equations give
$$\gamma_I(\tau(h_1, h_2))-\tau(\gamma_H(h_1), \gamma_H(h_2))=\nu_{\gamma_H(h_1
\circ h_2)}(\sigma_{\gamma_H(h_2)}(\nu^{-1}_{\gamma_H(h_1)}(\lambda(h_1))))-\lambda(h_1 \circ h_2)+\nu_{\gamma_H(h_1)}(\lambda(h_2)),$$
which, along with \eqref{modeqn2} proves  condition (ii), that is, $\gamma_I^*([(\beta, \tau)])= \gamma_H^*([(\beta, \tau)])$. 

Conversely, let $(\phi,\theta) \in \Autb(H) \times \Autb(I)$ satisfy conditions (i) and (ii). Condition (ii) guarantees the existence of a map  $\lambda : H \rightarrow I$ in $C_N^1$ such that   
$$\big(\beta^{(\phi, \Id)},\tau^{(\phi, \Id)}\big) - \big(\beta^{(\Id, \theta^{-1})},\tau^{(\Id, \theta^{-1})}\big) = \partial^1(\lambda).$$ 
Define the map $\gamma: E \rightarrow E$, for all $h \in H$ and $y \in I$,  by 
$$\gamma(s(h)+y) = s(\phi(h))+ \lambda(h) + \theta(y).$$
A routine check   then shows that $\gamma \in \Autb_I(E)$, $\gamma_H = \phi$ and $\gamma_I = \theta$. The proof is now complete. \hfill $\Box$

\end{proof}

We conclude with a reduction argument on extension and lifting problem for automorphisms of  skew  braces of nilpotent type, i.e., $(H, +)$ is nilpotent. We'll only deal with brace extensions from $\STExt_{(\nu, \mu, \sigma)}(H, I)$, with $H$ of finite order, for a given good triplet of actions $(\nu, \mu, \sigma)$ of $H$ on $I$.  Let us start with the following fixed extension of a finite left skew brace $H$ by a trivial brace $I$ lying in $\STExt_{(\nu, \mu, \sigma)}(H, I)$:
$$\mathcal{E} : 0 \to I \rightarrow E  \rightarrow  H \to 0.$$
 For this extension $\mathcal{E}$, by Theorem \ref{wells5 sb}, we get the exact sequence of groups
\begin{equation}\label{lift1}
0 \rightarrow \Autb^{H,I}(E) \rightarrow \Autb_I(E) \stackrel{\rho(\mathcal{E})}{\longrightarrow} \C_{(\nu, \mu, \sigma)} \stackrel{\omega(\mathcal{E})}{\longrightarrow} \RH^2_N(H,I),
\end{equation}
where $\omega(\mathcal{E})$ is a derivation.

Let $P_i$ denote a Sylow $p_i$-subgroup of $(H, \circ)$, where $p_i$ is a prime divisor of the order of $H$. Notice that $P_i$ is also  the  Sylow $p_i$-subgroup of $(H, +)$; hence, $H$ being nilpotent type,  $P_i$ becomes a left  ideal of $H$ (\cite[Lemma 4.10]{CSV19}). Let $R_i$ denote the pre-image of $P_i$ in $E$. Notice that $(\nu^i, \mu^i, \sigma^i)$ is a good triplet of actions of $P_i$ on $I$, where $\nu^i := \nu|_{P_i}$,  $\mu^i := \mu|_{P_i}$ and $\sigma^i := \sigma|_{P_i}$, and the extension 
$$\mathcal{E}_i : 0 \to I \to R_i \to P_i \to 0$$
lies in $\STExt_{(\nu^i, \mu^i, \sigma^i)}(H, I)$.  
Set 
$$\C_{(\nu^i, \mu^i, \sigma^i)} := \{ (\phi,  \theta) \in  \Autb(P_i) \times \Autb(I) \mid \nu^i_h=\theta^{-1}\nu^i_{\phi(h)}\theta, \mu^i_h=\theta^{-1}\mu^i_{\phi(h)}\theta \mbox{ and } \sigma^i_h = \theta^{-1}\sigma^i_{\phi(h)}\theta \}.$$
The extension $\mathcal{E}_i$ now gives  the exact sequence of groups
$$0 \rightarrow \Autb^{P_i,I}(R_i) \rightarrow \Autb_{I}(R_i) \stackrel{\rho(\mathcal{E}_i)}{\longrightarrow} \C_{(\nu^i, \mu^i, \sigma^i)} \stackrel{\omega(\mathcal{E}_i)}{\longrightarrow} \RH^2_N(P_i,I).$$

On the other hand, denote by $\Autb_{I, R_i}(E, I)$ the subgroup of $\Autb_I(H,I)$ consisting of all automorphism of $E$ normalising $I$ as well as $R_i$, and set $\C^i_{(\nu, \mu, \sigma)} := \{(\phi, \theta) \in \C_{(\nu, \mu, \sigma)} \mid \phi(P_i) = P_i\}$. Then for the extension
$$\mathcal{E} : 0 \to I \to E \to H \to 0,$$
 we get the following exact sequence of groups from \eqref{lift1}:
$$0 \rightarrow \Autb^{H,I}(E) \rightarrow \Autb_{I, R_i}(E) \stackrel{\rho(\mathcal{E})}{\longrightarrow} \C^i_{(\nu, \mu, \sigma)} \stackrel{\omega(\mathcal{E})}{\longrightarrow} \RH^2_N(H,I).$$
Let $\res^H_{P_i} : \RH^2_N(H,I) \to \RH^2_N(P_i,I)$ be the restriction homomorphisms as defined in Corollary \ref{cor-sec4}. Define $r^H_{P_i} : \C^i_{(\nu, \mu, \sigma)} \to \C_{(\nu^i, \mu^i, \sigma^i)}$ by 
$$r^H_{P_i} (\phi, \theta) = (\phi|_{P_i}, \theta).$$
Using the definition of $\omega(\mathcal{E})$, we now get the following commutative diagram:
\begin{equation}\label{lift2}
\begin{CD}
 \C^i_{(\nu, \mu, \sigma)}   @>{\omega(\mathcal{E})}>>  \RH^2_N(H, I)\\
  @VV{r^H_{P_i}}V  @VV{\res^H_{P_i}}V\\
 \C_{(\nu^i, \mu^i, \sigma^i)}  @>{\omega(\mathcal{E}_i)}>>  \RH^2_N(P_i, I).
\end{CD}
\end{equation}

Recall that $\Ho^2_N((H, \circ), I)$ denotes the second cohomology group of the group $(H, \circ)$ with coefficients in $I$, where the right action of $H$ on $I$ is through $\sigma$. Similarly $\Ho^2_N((P_i, \circ), I)$ denotes the second cohomology group of $(P_i, \circ)$ with coefficients in $I$, where the right action of $P_i$ on $I$ is through $\sigma^i = \sigma|_{P_i}$. By Proposition \ref{prop-sec4} there exist embeddings $\iota : \RH^2_N(H, I)  \to \Ho^2_N((H, \circ), I)$ and $\iota_i : \RH^2_N(P_i, I)  \to \Ho^2_N((P_i, \circ), I)$. We now get the following commutative diagram:
\begin{equation}\label{lift3}
\begin{CD}
\RH^2_N(H, I)  @>{\iota}>>  \Ho^2_N((H, \circ), I)\\
  @VV{\res^H_{P_i}}V  @VV{\res^H_{P_i}}V\\
\RH^2_N(P_i, I) @>{\iota_i}>>  \Ho^2_N((P_i, \circ), I).
\end{CD}
\end{equation}

Let $\pi(H) := \{p_1, \ldots, p_r\}$ be the set of all distinct prime divisors of the order of $H$. Please note the ad-hoc use of $\pi$. With this set-up, we finally have 
\begin{thm}
 Let $(\phi, \theta) \in \C_{(\nu, \mu, \sigma)}$ be such that $(\phi|_{P_i}, \theta) \in \C_{(\nu^i, \mu^i, \sigma^i)}$ is inducible for some Sylow $p_i$-subgroup of $H$ for each $p_i \in \pi(H)$. Then $(\phi, \theta)$ is inducible. 

Conversely, if $(\phi, \theta) \in \C_{(\nu, \mu, \sigma)}$ is inducible, then $(\phi|_{P_i}, \theta) \in \C_{(\nu^i, \mu^i, \sigma^i)}$ is inducible whenever $\phi(P_i) = P_i$.
\end{thm}
\begin{proof}
Let $(\phi, \theta) \in \C_{(\nu, \mu,  \sigma)}$ be as in the statement and $\omega(\mathcal{E})(\phi, \theta) = [f]$. Then by the hypothesis and the commutativity of  diagram \eqref{lift2}, it follows that $\res^H_{P_i}([f])$ is the identity element of $\RH^2_N(P_i, I)$. Further, by the commutativity of diagram \eqref{lift3}, we have $\res^H_{P_i}([f'])$ is the identity  element of $\Ho^2_N((P_i, \circ), I)$. Being the cohomology of groups, we can now use corestriction-restriction homomorphism result \cite[(9.5) Proposition. (ii)]{KSB} to deduce that $k[f']$ is the zero cohomology class of $\Ho^2_N((H, \circ), I)$, where $k$ is the index of $P_i$ in $H$. Since this is true for at least one Sylow $p_i$-subgroup for each $p_i \in \pi(H)$, it follows that $[f']$ is the identity element of $\Ho^2_N((H, \circ), I)$, which, $\iota$ being an embedding, implies that $[f]$ is  the zero cohomology class of $\Ho^2_N(H, I)$. Hence $(\phi, \theta)$ is inducible.

The converse part is left as an exercise for the reader. \hfill $\Box$ 

\end{proof}

\end{document}